\documentclass[11pt,leqno,twoside]{article}

\usepackage{amssymb}
\usepackage{amsmath}
\usepackage{amsthm}

\allowdisplaybreaks

\def\ls{\lesssim}
\def\gs{\gtrsim}
\def\fz{\infty}

\def\r{\right}
\def\lf{\left}

\def\supp{{\mathop\mathrm{\,supp\,}}}

\def\aa{{\mathbb A}}
\def\rr{{\mathbb R}}
\def\rh{{\mathbb R}{\mathbb H}}
\def\rn{{{\rr}^n}}
\def\zz{{\mathbb Z}}
\def\nn{{\mathbb N}}
\def\cc{{\mathbb C}}

\newcommand{\wz}{\widetilde}
\newcommand{\oz}{\overline}

\newcommand{\ca}{{\mathcal A}}

\newcommand{\cm}{{\mathcal M}}

\newcommand{\cp}{{\mathcal P}}

\newcommand{\ccr}{{\mathcal R}}
\newcommand{\cs}{{\mathcal S}}

\def\az{\alpha}
\def\lz{\lambda}
\def\blz{\Lambda}

\def\bfai{\Phi}
\def\dz {\delta}

\def\epz{\epsilon}

\def\bz{\beta}

\def\fai{\varphi}
\def\gz{{\gamma}}
\def\bgz{{\Gamma}}
\def\vz{\varphi}
\def\tz{\theta}

\def\wz{\widetilde}

\def\ls{\lesssim}
\def\gs{\gtrsim}

\def\boz{\Omega}
\def\oz{\omega}

\def\uc{{\varepsilon}}

\def\fin{{\mathop\mathrm{fin}}}

\def\esup{\mathop\mathrm{\,ess\,sup\,}}

\def\bmo{{{\mathop\mathrm{bmo}}}}
\def\bbmo{{{\mathop\mathrm{BMO}}}}

\def\fin{{{\mathop\mathrm{fin}}}}
\def\hs{\hspace{0.3cm}}

\def\dint{\displaystyle\int}

\def\dsup{\displaystyle\sup}

\newtheorem{theorem}{Theorem}[section]
\newtheorem{lemma}[theorem]{Lemma}
\newtheorem{corollary}[theorem]{Corollary}
\newtheorem{proposition}[theorem]{Proposition}
\theoremstyle{definition}
\newtheorem{remark}[theorem]{Remark}
\newtheorem{definition}[theorem]{Definition}
\numberwithin{equation}{section}
\def\supp{{\mathop\mathrm{\,supp\,}}}

\def\dist{{\mathop\mathrm{\,dist\,}}}
\def\loc{{\mathop\mathrm{loc\,}}}
\def\lfz{\lfloor}
\def\rfz{\rfloor}
\numberwithin{equation}{section}

\begin{document}

\arraycolsep=1pt

\title{\bf\Large Lusin Area Function and Molecular Characterizations of
Musielak-Orlicz Hardy Spaces
and Their Applications\footnotetext {\hspace{-0.35cm}
2010 {\it Mathematics Subject Classification}. Primary: 42B25;
Secondary: 42B30, 42B35, 46E30.
\endgraf {\it Key words and phrases}. Musielak-Orlicz
function, Hardy space, atom, molecule,
Lusin area function, $\mathop\mathrm{BMO}$
space, Carleson measure.
\endgraf  Dachun Yang and Sibei Yang are partially supported
by 2010 Joint Research Project Between
China Scholarship Council and German Academic Exchange
Service (PPP) (Grant No. LiuJinOu [2010]6066).
Dachun Yang is also partially supported by the National
Natural Science Foundation  of China (Grant No. 11171027) and
the Specialized Research Fund for the Doctoral Program of Higher Education
of China (Grant No. 20120003110003). Part of this paper was finished
during the course of the visit of Dachun Yang and Sibei Yang to
the Mathematisches Institut of Friedrich-Schiller-Universit\"at Jena
and they want to express their sincere
and deep thanks for the gracious hospitality of the the Research group
``Function spaces" therein.}}
\author{Shaoxiong Hou, Dachun Yang\footnote{Corresponding author}\, \ and Sibei Yang}
\date{ }
\maketitle

\vspace{-0.5cm}

\begin{center}
\begin{minipage}{13.5cm}
{\small {\bf Abstract}\quad
Let $\varphi: \mathbb R^n\times
[0,\infty)\to[0,\infty)$ be a growth function such that
$\varphi(x,\cdot)$ is nondecreasing, $\varphi(x,0)=0$,
$\varphi(x,t)>0$ when $t>0$, $\lim_{t\to\infty}\varphi(x,t)=\infty$, and $\varphi(\cdot,t)$ is
a Muckenhoupt $A_\infty(\mathbb{R}^n)$ weight uniformly in $t$. In this paper, the
authors establish the Lusin
area function and the molecular characterizations of
the Musielak-Orlicz Hardy space $H_\varphi(\mathbb{R}^n)$
introduced by Luong Dang Ky via the grand maximal function.
As an application, the authors obtain
the $\varphi$-Carleson measure characterization of the Musielak-Orlicz
${\mathop\mathrm{BMO}}$-type space
$\mathop\mathrm{BMO}_{\varphi}(\mathbb{R}^n)$, which was proved to be the dual space
of $H_\varphi(\mathbb{R}^n)$ by Luong Dang Ky.}
\end{minipage}
\end{center}

\section{Introduction\label{s1}}

\hskip\parindent The real-variable theory of
Hardy spaces on the $n$-dimensional Euclidean space $\rn$ was originally
studied by Stein and Weiss \cite{sw60} and systematically
developed by Fefferman and Stein in a seminal paper \cite{fs72}.
Since the Hardy space $H^p(\rn)$ with $p\in(0,1]$ is, especially when studying
the boundedness of operators, a suitable
substitute of the Lebesgue space $L^p(\rn)$,
it plays an important role in various fields of
analysis and partial differential equations (see, for example, \cite{clms,st93}
and their references). In order to conveniently apply the real-variable theory
of $H^p(\rn)$ with $p\in(0,1]$,
their several equivalent characterizations
were revealed one after the other (see, for example, \cite{fs72,c74,l78,tw80}).
Among others a very important and useful
characterization of Hardy spaces is their atomic characterizations,
which were obtained by Coifman \cite{c74} when $n=1$ and
Latter \cite{l78} when $n>1$. Later, as an extension of this characterization,
the molecular characterization of Hardy spaces was established
by Taibleson and Weiss \cite{tw80}.

On the other hand, due to the need for more inclusive classes of function spaces
than the $L^p(\rn)$-families from applications,
Orlicz spaces were introduced by Birnbaum-Orlicz in \cite{bo31}
and Orlicz in \cite{o32}, which is widely used in various branches
of analysis (see, for example, \cite{rr91,rr00} and their references).
Moreover, as a development of the theory of
Orlicz spaces, Orlicz-Hardy spaces and their
dual spaces were studied by Str\"omberg \cite{s79}
and Janson \cite{ja80} on $\rn$ and,
quite recently, Orlicz-Hardy spaces associated with divergence form elliptic
operators by Jiang and Yang \cite{jy10}.

Furthermore, the classical $\bbmo$ space (the \emph{space
of functions with bounded mean oscillation}), originally
introduced by John and Nirenberg \cite{jn1}, plays an important
role in the study of partial differential equations and harmonic
analysis. In particular, Fefferman and Stein \cite{fs72} proved that
$\bbmo(\rn)$ is the dual space of $H^1 (\rn)$ and also obtained the Carleson
measure characterization of $\bbmo(\rn)$. Moreover, the generalized $\bbmo$-type
space $\bbmo_\rho(\rn)$ was studied in \cite{s79,ja80,hsv07} and it was proved
therein to be the dual space of the Orlicz-Hardy space
$H_\Phi(\rn)$, where the function $\Phi:\, [0,\fz)\to[0,\fz)$
satisfies the following assumptions:
\begin{equation}\label{o}
\Phi\ \text{is nondecreasing},\ \Phi(0)=0,\
\Phi(t)>0\ \text{when}\ t>0,\ \text{and}\ \lim_{t\to\fz}\Phi(t)=\fz,
\end{equation}
and $\rho(t):=t^{-1}/\Phi^{-1}(t^{-1})$ for all
$t\in(0,\fz)$. Here and in what follows, $\Phi^{-1}$ denotes the
\emph{inverse function} of $\Phi$. Observe that $\Phi$ may not be
convex and hence may not be an Orlicz function in the classical sense.
Meanwhile, the Carleson measure
characterization of $\bbmo_\rho(\rn)$ was obtained in \cite{hsv07}.

Recently, a new Musielak-Orlicz Hardy space
$H_{\fai}(\rn)$ was introduced by Ky \cite{k}, via the grand maximal
function, which includes both the Orlicz-Hardy space in
\cite{s79,ja80} and the
weighted Hardy space $H^p_\oz(\rn)$ with $p\in(0,1]$ and
$\oz\in A_{\fz}(\rn)$ in \cite{g79,st}. Here and in what follows,
$\fai:\,\rn\times[0,\fz)\to[0,\fz)$ is a growth function such
that $\fai(x,\cdot)$, for any fixed $x\in\rn$, satisfies \eqref{o} with uniformly upper type
1 and lower type $p$ for some $p\in(0,1]$ (see Section \ref{s2} for the
definitions of uniformly upper or lower types), and $\fai(\cdot,t)$ is
a Muckenhoupt $A_\fz(\rn)$ weight uniformly in $t$, and $A_q(\rn)$ with $q
\in[1,\fz]$ denotes the
\emph{class of Muckenhoupt weights} (see, for example,
\cite{gra1} for their definitions and properties).
In \cite{k}, Ky first established the atomic characterization of $H_{\fai}(\rn)$ and
further introduced the Musielak-Orlicz $\mathop\mathrm{BMO}$-type space
$\mathop\mathrm{BMO}_{\fai}(\rn)$, which was proved to be the dual
space of $H_{\fai}(\rn)$. Furthermore, some interesting applications of
these spaces were also presented in \cite{bfg10,bgk,bijz07,k,k2,k1}.
Moreover, the local Musielak-Orlicz Hardy space, $h_\fai(\rn)$, and its
dual space, $\bmo_\fai(\rn)$,
were studied in \cite{yys} and some applications of $h_\fai(\rn)$
and $\bmo_\fai(\rn)$, to pointwise multipliers of
$\bbmo$-type spaces and to
the boundedness of local Riesz transforms and pseudo-differential
operators on $h_\fai(\rn)$, were also obtained in \cite{yys}.
Recall that Musielak-Orlicz functions are the
natural generalization of Orlicz functions that may vary in the
spatial variables (see, for example, \cite{k,m83}).
Moreover, the motivation to study function spaces of Musielak-Orlicz type
is due to that they have wide
applications in several branches of physics and mathematics
(see, for example, \cite{bg10,bgk,bijz07,k,yys}
for more details).

Motivated by \cite{k,tw80}, in this paper, we establish the Lusin
area function and the molecular
characterizations of the Musielak-Orlicz Hardy space $H_\varphi(\mathbb{R}^n)$.
As an application, we obtain the $\varphi$-Carleson measure characterization
of the Musielak-Orlicz $\bbmo$-type space
$\mathop\mathrm{BMO}_{\varphi}(\mathbb{R}^n)$.

Precisely, this paper is organized as follows. In Section \ref{s2},
we recall some notions and examples of growth functions,
as well as their properties
established in \cite{k}.

In Section \ref{s3}, we first recall some notions about tent
spaces and then study the Musielak-Orlicz tent space $T_\fai(\rr^{n+1}_+)$.
The main target of this section is to establish the
atomic characterization of $T_\fai(\rr^{n+1}_+)$ (see Theorem \ref{t3.1} below).
As a byproduct, we show that, if $f\in T_\fai(\rr^{n+1}_+)\cap T_2^p(\rr^{n+1}_+)$
with $p\in(0,\fz)$, then the atomic decomposition of $f$ holds in both
$T_\fai(\rr^{n+1}_+)$ and $T_2^p(\rr^{n+1}_+)$, which plays an important role in
the remainder of this paper (see Corollary \ref{c3.1} below). Also, a subtle
observation on the atomic decomposition for functions in $T_\fai(\rr^{n+1}_+)$
is presented in Remark \ref{r3.x1} below, which is needed
in the proof of Proposition \ref{p4.2}
below.

In Section \ref{s4}, we introduce the Hardy-type spaces,
$H_{\fai,S}(\rn)$ and $H^{q,s,\uc}_{\fai,\mathrm{mol}}(\rn)$, respectively,
via the Lusin area function and the molecule, and then prove that the
operator $\pi_{\phi}$, which was first introduced in \cite{cms85} (see also
\eqref{4.2} below), maps the Musielak-Orlicz tent space $T_\fai(\rr^{n+1}_+)$
continuously into $H_{\fai,S}(\rn)$
(see Proposition \ref{p4.1} below). By this and the atomic
decomposition of $T_\fai(\rr^{n+1}_+)$, we conclude
that, for each $f\in H_{\fai,S}(\rn)$ vanishing weakly at infinity (see Section
\ref{s4} below for its definition),
there exists a molecular decomposition of $f$ holding in both $\cs'(\rn)$
(the space of Schwartz distributions) and $H_{\fai,S}(\rn)$
(see Proposition \ref{p4.2} below). Via this molecular decomposition
of $H_{\fai,S}(\rn)$ and the atomic characterization of
$H_\fai(\rn)$ established by Ky \cite{k}, we further obtain the Lusin area function
and the molecular characterizations of $H_\fai(\rn)$ (see Theorem \ref{t4.1} below).

In Section \ref{s5}, we first recall the definition of the Musielak-Orlicz
$\bbmo$-type space
$\mathrm{BMO}_{\fai}(\rn)$ and introduce the $\fai$-Carleson
measure. When $\fai$ further satisfies $nq(\fai)<(n+1)i(\fai)$ and $q(\fai)r(\fai)/
[r(\fai)-1]\in(1,2)$ (see \eqref{2.3}, \eqref{2.1} and \eqref{2.4} below,
respectively, for the definitions of $q(\fai)$, $i(\fai)$ and $r(\fai)$),
then in Theorem \ref{t5.1} below,
we establish the $\fai$-Carleson measure characterization
of $\mathrm{BMO}_{\fai}(\rn)$ by using the Lusin area function characterization
of $H_{\fai}(\rn)$ in Theorem \ref{t4.1} and an equivalent characterization of
the space $\bbmo_\fai(\rn)$ obtained in \cite[Theorem 2.7]{ly13} (see also Lemma
\ref{l5.1} below).

We remark that the method for obtaining the Lusin area function characterization
of $H_\fai(\rn)$ in this paper is different from the method used in \cite{fs}.
More precisely, in \cite{fs}, the Lusin area function
characterization of Hardy spaces was established by using the Calder\'on
reproducing formula and a subtle decomposition of all dyadic cubes
in $\rn$. However, in this paper,
we establish the Lusin area function characterization
of $H_\fai(\rn)$ by using the Calder\'on reproducing formula
(see \eqref{4.22} below), the atomic decomposition of the Musielak-Orlicz tent space
in Theorem \ref{t3.1} and some boundedness of the operator $\pi_\phi$ in Proposition
\ref{p4.1}. This method is closer to the method used in
\cite{cms85,jy10,sy10}. Moreover, different from
\cite{sy10}, we do not need the additional assumption
that, for any $t\in[0,\fz)$, $\fai(\cdot,t)$ satisfies
the reverse H\"older inequality of order $2$ (see Definition \ref{d2.1}
below for its definition), by
fully using the $L^p(\rn)$ boundedness of the Lusin area function $S$ for all $p\in(1,\fz)$.

By using the Lusin area function
characterization of $H_\fai(\rn)$ obtained in this article,
Liang, Huang and Yang \cite{lhy}, via establishing a
Musielak-Orlicz Fefferman-Stein vector-valued inequality,
further established the Littlewood-Paley $g$-function
and $g^\ast_\lz$-function characterizations of $H_\fai(\rn)$
with $\fai$ satisfying the same assumptions as in this article. Furthermore,
the characterizations of $H_\fai(\rn)$ in terms of the vertical and the non-tangential
maximal functions were also obtained in \cite{lhy}.

We also point out that the main results of this article,
including the Lusin area function and the molecular characterizations of
$H_\varphi(\mathbb{R}^n)$ and the $\fai$-Carleson measure characterization
of $\mathrm{BMO}_{\fai}(\rn)$, have local variants,
which will be studied in a forthcoming article (see \cite{yys} for
the definition of the local Musielak-Orlicz Hardy space $h_\fai(\rn)$).

Finally we make some conventions on notation. Throughout the whole
article, we denote by $C$ a \emph{positive constant} which is
independent of main parameters, but it may vary from line to
line. We also use $C(\gz,\bz,\ldots)$ to denote a \emph{positive
constant depending on the indicated parameters $\gz$, $\bz$,
$\ldots$}. The \emph{symbol} $A\ls B$ means that $A\le CB$. If
$A\ls B$ and $B\ls A$, then we write $A\sim B$. The  \emph{symbol}
$\lfz s\rfz$ for $s\in\rr$ denotes the maximal integer $k$ such that
$k\le s$. For any given normed spaces $\mathcal A$ and $\mathcal
B$ with the corresponding norms $\|\cdot\|_{\mathcal A}$ and
$\|\cdot\|_{\mathcal B}$, the \emph{symbol} ${\mathcal
A}\hookrightarrow{\mathcal B}$ means that, for all $f\in \mathcal A$, then
$f\in\mathcal B$ and $\|f\|_{\mathcal B}\ls \|f\|_{\mathcal A}$.
For any subset $E$ of $\rn$, we denote by $E^\complement$ the
\emph{set} $\rn\setminus E$ and by
$\chi_{E}$  its \emph{characteristic function}. We also set
$\nn:=\{1,\,2,\, \ldots\}$ and $\zz_+:=\{0\}\cup\nn$. For any
$\tz:=(\tz_{1},\ldots,\tz_{n})\in\zz_{+}^{n}$, let
$|\tz|:=\tz_{1}+\cdots+\tz_{n}$ and
$\partial^{\tz}_x:=\frac{\partial^{|\tz|}}{\partial
{x_{1}^{\tz_{1}}}\cdots\partial {x_{n}^{\tz_{n}}}}.$
For any index $q\in[1,\fz]$, we denote by $q'$ its
\emph{conjugate index}, namely, $1/q+1/q'=1$.

\section{Growth functions\label{s2}}

\hskip\parindent In this section, we first
recall some notions and assumptions on growth functions
considered in this article and give some examples which satisfy these assumptions.
We also recall some properties of growth functions established in \cite{k}.

Let $\Phi:\,[0,\fz)\to[0,\fz)$ be as in \eqref{o}. We say that the function $\Phi$ is of
\emph{upper type $p$} (resp. \emph{lower type $p$}) for some $p\in[0,\fz)$, if
there exists a positive constant $C$ such that, for all
$t\in[1,\fz)$ (resp. $t\in[0,1]$) and $s\in[0,\fz)$,
$\Phi(st)\le Ct^p \Phi(s).$

For a given function $\fai:\,\rn\times[0,\fz)\to[0,\fz)$ such that, for
any $x\in\rn$, $\fai(x,\cdot)$ satisfies \eqref{o},
we say that $\fai$ is of \emph{uniformly upper type $p$} (resp.
\emph{uniformly lower type $p$}) for some $p\in[0,\fz)$, if there
exists a positive constant $C$ such that, for all $x\in\rn$,
$t\in[0,\fz)$ and $s\in[1,\fz)$ (resp. $s\in[0,1]$), $\fai(x,st)\le Cs^p\fai(x,t)$.
We say that $\fai$ is of \emph{positive uniformly upper type}
(resp. \emph{uniformly lower type}), if it is of uniformly upper
type (resp. uniformly lower type) $p$ for some $p\in(0,\fz)$. Let
\begin{equation}\label{2.1}
i(\fai):=\sup\{p\in(0,\fz):\ \fai\ \text{is of uniformly lower
type}\ p\}.
\end{equation}
Observe that $i(\fai)$ may not be attainable, namely, $\fai$ may not be of uniformly
lower type $i(\fai)$; see below for some examples.

\begin{definition}\label{d2.1}
We say that the function $\fai(\cdot,t)$ satisfies the
\emph{uniformly Muckenhoupt condition for some $q\in[1,\fz)$},
denoted by $\fai\in\aa_q(\rn)$, if, when $q\in (1,\fz)$,
\begin{equation}\label{2.2}
\aa_q (\fai):=\sup_{t\in
(0,\fz)}\sup_{B\subset\rn}\frac{1}{|B|^q}\int_B
\fai(x,t)\,dx \lf\{\int_B
[\fai(y,t)]^{-q'/q}\,dy\r\}^{q/q'}<\fz,
\end{equation}
where $1/q+1/q'=1$, or
\begin{equation*}
\aa_1 (\fai):=\sup_{t\in (0,\fz)}
\sup_{B\subset\rn}\frac{1}{|B|}\int_B \fai(x,t)\,dx
\lf(\esup_{y\in B}[\fai(y,t)]^{-1}\r)<\fz.
\end{equation*}
Here the first supremums are taken over all $t\in(0,\fz)$ and the
second ones over all balls $B\subset\rn$.

We say that the function $\fai(\cdot,t)$ satisfies the
\emph{uniformly reverse H\"older condition for some
$q\in(1,\fz]$}, denoted by $\fai\in \rh_q(\rn)$, if, when $q\in
(1,\fz)$,
\begin{eqnarray*}
\rh_q (\fai):&&=\sup_{t\in (0,\fz)}\sup_{B\subset\rn}\lf\{\frac{1}
{|B|}\int_B [\fai(x,t)]^q\,dx\r\}^{1/q}\lf\{\frac{1}{|B|}\int_B
\fai(x,t)\,dx\r\}^{-1}<\fz,
\end{eqnarray*}
or
\begin{equation*}
\rh_{\fz} (\fai):=\sup_{t\in
(0,\fz)}\sup_{B\subset\rn}\lf\{\esup_{y\in
B}\fai(y,t)\r\}\lf\{\frac{1}{|B|}\int_B
\fai(x,t)\,dx\r\}^{-1} <\fz.
\end{equation*}
Here the first supremums are taken over all $t\in(0,\fz)$ and the
second ones over all balls $B\subset\rn$.
\end{definition}

Recall that, in Definition \ref{d2.1},
$\aa_q(\rn)$ with $q\in[1,\fz)$ was introduced by Ky \cite{k}.

Let $\aa_{\fz}(\rn):=\cup_{q\in[1,\fz)}\aa_{q}(\rn)$ and define
the \emph{critical indices} of $\fai\in\aa_{\fz}(\rn)$ as follows:
\begin{equation}\label{2.3}
q(\fai):=\inf\lf\{q\in[1,\fz):\ \fai\in\aa_{q}(\rn)\r\}
\end{equation}
and
\begin{equation}\label{2.4}
r(\fai):=\sup\lf\{q\in(1,\fz]:\ \fai\in\rh_{q}(\rn)\r\}.
\end{equation}

Now we introduce the notion of growth functions.

\begin{definition}\label{d2.2}
We say that a function $\fai:\rn\times[0,\fz)\rightarrow[0,\fz)$ is
a \emph{growth function}, if the following hold:
 \vspace{-0.25cm}
\begin{enumerate}
\item[(i)] $\fai$ is a \emph{Musielak-Orlicz function}, namely,
\vspace{-0.2cm}
\begin{enumerate}
    \item[(i)$_1$] the function $\fai(x,\cdot):[0,\fz)\to[0,\fz)$,
    for any fixed $x\in\rn$, satisfies \eqref{o};
    \vspace{-0.2cm}
    \item [(i)$_2$] the function $\fai(\cdot,t)$ is a measurable
    function for any fixed $t\in[0,\fz)$.
\end{enumerate}
\vspace{-0.25cm} \item[(ii)] $\fai\in \aa_{\fz}(\rn)$.
\vspace{-0.25cm} \item[(iii)] The function $\fai$ is of positive
uniformly lower type $p$ for some $p\in(0,1]$ and of uniformly
upper type 1.
\end{enumerate}
\end{definition}

Clearly, $\fai(x,t):=\oz(x)\Phi(t)$ is a growth function if
$\oz\in A_{\fz}(\rn)$ and $\Phi$ is as in \eqref{o} with lower
type $p$ for some $p\in(0,1]$ and upper type 1. It is known
that, for $p\in(0,1]$, if $\Phi(t):=t^p$ for all $t\in [0,\fz)$,
then $\Phi$ satisfies \eqref{o} and $\Phi$ is of lower type $p$ and also upper type $p$.
For $p\in[\frac{1}{2},1]$, if
$\Phi(t):= t^p/\ln(e+t)$ for all $t\in [0,\fz)$, then $\Phi$ satisfies \eqref{o}
and $\Phi$ is of lower type $q$ for $q\in(0, p)$ and of upper type $p$.
Observe that the same conclusions also hold true
for the function $\Phi(t):= t^p/\ln(c_p+t)$
for all $t\in [0,\fz)$ when $p\in(0,\frac{1}{2})$,
where $c_p$ is a positive constant large enough, depending on $p$, such that
$\Phi$ is nondecreasing on $[0,\fz)$.
For $p\in(0,1]$, if $\Phi(t):=t^p\ln(e+t)$ for all $t\in
[0,\fz)$, then $\Phi$ satisfies \eqref{o}
and $\Phi$ is of lower type $p$ and of upper type $q$
for $q\in(p,1]$. Recall that if a function satisfying \eqref{o} is of upper
type $p\in(0,1)$, then it is also of upper type 1.

Another typical and useful growth
function is
\begin{equation}\label{2.4a}
\fai(x,t):=\frac{t^{\az}}{[\ln(e+|x|)]^{\bz}+[\ln(e+t)]^{\gz}}
\end{equation}
for all $x\in\rn$ and $t\in[0,\fz)$ with any $\az\in(0,1]$,
$\bz\in[0,\fz)$ and $\gz\in [0,2\az(1+\ln2)]$; more precisely,
$\fai\in \aa_1(\rn)$, $\fai$ is of uniformly upper type $\az$ and
$i(\fai)=\az$ which is not attainable (see \cite{k}).
Observe also that, when $\gz\in (2\az(1+\ln2),\fz)$, the same conclusions
hold true for the function $\fai(x,t):=\frac{t^{\az}}{[\ln(e+|x|)]^{\bz}+[\ln(c_\gz+t)]^{\gz}}$
for all $x\in\rn$ and $t\in[0,\fz)$, where $c_\gz$ is a positive constant
large enough, depending on $\gz$, such that
$\fai$ is nondecreasing on the
time variables $t$.

Throughout the whole article, we always
assume that $\fai$ is a growth function as in Definition
\ref{d2.2}. Let us now introduce the Musielak-Orlicz space.

The \emph{Musielak-Orlicz space $L^{\fai}(\rn)$} is defined as the set
of all measurable functions $f$ such that
$\int_{\rn}\fai(x,|f(x)|)\,dx<\fz$ with the \emph{Luxembourg
norm}
$$\|f\|_{L^{\fai}(\rn)}:=\inf\lf\{\lz\in(0,\fz):\ \int_{\rn}
\fai\lf(x,\frac{|f(x)|}{\lz}\r)\,dx\le1\r\}.
$$
In what follows, for any measurable subset $E$ of $\rn$ and $t\in[0,\fz)$, we let
$$\fai(E,t):=\int_E\fai(x,t)\,dx.$$

The following lemma, which gives the
properties of growth functions,  is just \cite[Lemmas
4.1 and 4.2]{k}.

\begin{lemma}\label{l2.1}
{\rm(i)} Let $\fai$ be a growth function. Then $\fai$ is uniformly
$\sigma$-quasi-subadditive on $\rn\times[0,\fz)$, namely, there
exists a positive constant $C$ such that, for all
$(x,t_j)\in\rn\times[0,\fz)$ with $j\in\nn$,
$\fai(x,\sum_{j=1}^{\fz}t_j)\le C\sum_{j=1}^{\fz}\fai(x,t_j).$

{\rm(ii)} Let $\fai$ be a growth function and
$\wz{\fai}(x,t):=\int_0^t\frac{\fai(x,s)}{s}\,ds$ for all
$(x,t)\in\rn\times[0,\fz)$. Then $\wz{\fai}$ is a growth function,
which is equivalent to $\fai$; moreover, $\wz{\fai}(x,\cdot)$ is
continuous and strictly increasing.

{\rm (iii)} Let $\fai$ be a growth function. Then $\int_{\mathbb
R^n}\varphi(x,\frac{|f(x)|}{\|f\|_{L^\varphi(\rn)}})\,dx=1$ for all $f\in
L^\varphi(\mathbb R^n)\setminus\{0\}$.

\end{lemma}

We have the following properties for $\aa_\fz(\rn)$, whose proofs
are similar to those in \cite{gra1}, the details being omitted.

\begin{lemma}\label{l2.4}
$\mathrm{(i)}$ $\aa_1(\rn)\subset\aa_p(\rn)\subset\aa_q(\rn)$ for
$1\le p\le q<\fz$.

$\mathrm{(ii)}$ $\rh_{\fz}(\rn)\subset\rh_p(\rn)\subset\rh_q(\rn)$
for $1<q\le p\le\fz$.

$\mathrm{(iii)}$ If $\fai\in\aa_p(\rn)$ with $p\in(1,\fz)$, then
there exists $q\in(1,p)$ such that $\fai\in\aa_q(\rn)$.

$\mathrm{(iv)}$ $\aa_{\fz}(\rn)=\cup_{p\in[1,\fz)}\aa_p(\rn)
=\cup_{q\in(1,\fz]}\rh_q(\rn)$.

$\mathrm{(v)}$ If $p\in(1,\fz)$ and $\fai\in \aa_{p}(\rn)$, then
there exists a positive constant $C$ such that, for all measurable
functions $f$ on $\rn$ and $t\in[0,\fz)$,
$$\int_{\rn}\lf[\cm(f)(x)\r]^p\fai(x,t)\,dx\le
C\int_{\rn}|f(x)|^p\fai(x,t)\,dx,$$ where $\cm$ denotes the
Hardy-Littlewood maximal function on $\rn$, defined by setting, for all $x\in\rn$,
$$\cm(f)(x):=\sup_{x\in B}\frac{1}{|B|}\int_B|f(y)|\,dy,$$
where the supremum is taken over all balls $B\ni x$.

$\mathrm{(vi)}$ If $\fai\in \aa_{p}(\rn)$ with $p\in[1,\fz)$, then
there exists a positive constant $C$ such that, for all balls
$B_1,\,B_2\subset\rn$ with $B_1\subset B_2$ and $t\in[0,\fz)$,
$\frac{\fai(B_2,t)}{\fai(B_1,t)}\le
C[\frac{|B_2|}{|B_1|}]^p.$

$\mathrm{(vii)}$ If $\fai\in \rh_{q}(\rn)$ with $q\in(1,\fz]$,
then there exists a positive constant $C$ such that, for all balls
$B_1,\,B_2\subset\rn$ with $B_1\subset B_2$ and $t\in[0,\fz)$,
$\frac{\fai(B_2,t)}{\fai(B_1,t)}\ge
C[\frac{|B_2|}{|B_1|}]^{(q-1)/q}.$
\end{lemma}

\section{Musielak-Orlicz tent spaces\label{s3}}

\hskip\parindent In this section, we study the tent spaces
associated with the growth function $\fai$ as in
Definition \ref{d2.2}. We first recall some notations as follows.

Let $\rr^{n+1}_+:=\rn\times(0,\fz)$. For any $x\in\rn$, let
$$\bgz(x):=\{(y,t)\in\rr^{n+1}_+:\ |x-y|<t\}
$$
be the \emph{cone} of aperture 1 with vertex $x\in\rn$. For any
closed set $F$ of $\rn$, denote by $\ccr F$ the union of all
cones with vertices in $F$ (namely, $\ccr F:=\cup_{x\in
F}\bgz(x)$) and, for any open set $O$ in $\rn$, the tent over
$O$ by $\widehat{O}$, which is defined as
$\widehat{O}:=[\ccr(O^\complement)]^{\complement}$. It is
easy to see that
$$\widehat{O}=\lf\{(x,t)\in\rr^{n+1}_+:\ d(x,O^\complement)
\ge t\r\}.$$

For all measurable functions $g$ on $\rr^{n+1}_+$ and $x\in\rn$, define
$$\ca(g)(x):=\lf\{\int_{\bgz(x)}|g(y,t)|^2
\frac{dy\,dt}{t^{n+1}}\r\}^{1/2}.
$$
We remark that Coifman, Meyer and Stein \cite{cms85} studied the
tent space $T^p_2(\rr^{n+1}_+)$ for $p\in(0,\fz)$. Recall that ones say that a
measurable function $g$ is in the \emph{tent space} $T^p_2(\rr^{n+1}_+)$ with $p\in(0,\fz)$,
if $\|g\|_{T^p_2(\rr^{n+1}_+)}:=\|\ca(g)\|_{L^p(\rn)}<\fz$. Moreover,
the tent space $T_{\bfai}(\rr^{n+1}_+)$ associated with the
function $\bfai$ satisfying \eqref{o} was studied in \cite{hsv07,jy10}.

Let $\fai$ be as in Definition \ref{d2.2}. In what follows, we denote by
$T_{\fai}(\rr^{n+1}_+)$ the space of all measurable functions
$g$ on $\rr^{n+1}_+$ such that $\ca(g)\in L^{\fai}(\rn)$ and, for any
$g\in T_{\fai}(\rr^{n+1}_+)$, we define its \emph{quasi-norm} by
$$\|g\|_{T_{\fai}(\rr^{n+1}_+)}:=\|\ca(g)\|_{L^{\fai}(\rn)}=
\inf\lf\{\lz\in(0,\fz):\
\int_{\rn}\fai\lf(x,\frac{\ca(g)(x)}{\lz}\r)\,dx\le1\r\}.$$

Let $p\in(1,\fz)$. We say that a function $a$ on $\rr^{n+1}_+$ is a
\emph{$(\fai,\,p)$-atom}, if

(i) there exists a ball $B\subset\rn$ such that $\supp(a)\subset\widehat{B}$;

(ii) $\|a\|_{T^p_2(\rr^{n+1}_+)}\le
|B|^{1/p}\|\chi_B\|_{L^\fai(\rn)}^{-1}$.

Furthermore, if $a$ is a $(\fai,p)$-atom for all $p\in (1,\fz)$, we then
say that $a$ is a \emph{$(\fai,\fz)$-atom}.

For $(\fai,\fz)$-atoms, we have the following conclusion.

\begin{lemma}\label{l3.1a}
Let $\fai$ be as in Definition \ref{d2.2}. Then for any $(\fai,\fz)$-atom $a$,
it holds that $a\in T_{\fai}(\rr^{n+1}_+)$. Moreover, there exists a positive constant
$C$ such that, for any $(\fai,\fz)$-atom $a$ with $\supp(a)\subset\widehat B$
and any $\lz\in(0,\fz)$,
\begin{equation}\label{3.1a}
\int_{\rn}\fai\lf(x,\frac{\ca(a)(x)}{\lz}\r)\,dx
\le C\fai\lf(B,\frac{1}
{\lz\|\chi_{B}\|_{L^\fai(\rn)}}\r)
\end{equation}
and, in particular, there exists a positive constant
$\wz{C}$, depending only on $C$, such that $\|a\|_{T_\fai(\rr^{n+1}_+)}\le \wz C.$
\end{lemma}

\begin{proof}
Let $a$ be as in Lemma \ref{l3.1a}.
Assume first that \eqref{3.1a} holds for a moment. By \eqref{3.1a} with $\lz=1$ and
Lemma \ref{l2.1}(iii), we see that there exists a positive constant
$\wz{C}$, depending only on the constant $C$ in \eqref{3.1a}, such that
\begin{equation*}
\int_\rn\fai\lf(x,\frac{\ca(a)(x)}{\wz{C}}\r)\,dx\le1,
\end{equation*}
which implies that $a\in T_{\fai}(\rr^{n+1}_+)$ and
$\|a\|_{T_\fai(\rr^{n+1}_+)}\le \wz{C}.$

Now we show \eqref{3.1a}. To this end, by $\supp (a)\subset\widehat{B}$, we see that
$\supp(\ca(a))\subset B$. Furthermore, by $\fai\in
\aa_{\fz}(\rn)$ and Lemma \ref{l2.4}(iv), we know that there exists $q_0\in(1,\fz)$
such that $\fai\in\rh_{q_0}(\rn)$. From this, the uniformly upper type 1
property of $\fai$, H\"older's inequality and that $a$ is a $(\fai,\,\fz)$-atom,
we deduce that
\begin{eqnarray*}
&&\int_{\rn}\fai\lf(x,\frac{\ca(a)(x)}{\lz}\r)\,dx\\
&&\hs\ls\int_{B}\lf[1+\ca(a)(x)\|\chi_B\|_{L^\fai(\rn)}\r]
\fai\lf(x,\frac{1}{\lz\|\chi_B\|_{L^\fai(\rn)}}\r)\,dx\\
&&\hs\ls\fai\lf(B,\frac{1}{\lz\|\chi_B\|_{L^\fai(\rn)}}\r)+\lf\{\int_{
B}\lf[\ca(a)(x)\r]^{q_0'}\,dx\r\}^{1/q_0'}
\|\chi_B\|_{L^\fai(\rn)}\\
&&\hs\hs\times\lf\{\int_B\lf[\fai\lf(x,\frac{1}{\lz\|\chi_B
\|_{L^\fai(\rn)}}\r)\r]^{q_0}\,dx\r\}^{1/q_0}\\
&&\hs\ls\fai\lf(B,\frac{1}{\lz\|\chi_B\|_{L^\fai(\rn)}}\r)
+\|a\|_{T^{q_0'}_2(\rr^{n+1}_+)}\|\chi_B\|_{L^\fai(\rn)}\\
&&\hs\hs\times|B|^{-1/q_0'}\fai\lf(B,\frac{1}{\lz\|\chi_B\|_{L^\fai(\rn)}}\r)\ls
\fai\lf(B,\frac{1}{\lz\|\chi_B\|_{L^\fai(\rn)}}\r).
\end{eqnarray*}
Thus, \eqref{3.1a} holds true, which completes
the proof of Lemma \ref{l3.1a}.
\end{proof}

For functions in the space $T_{\fai}(\rr^{n+1}_+)$, we have the following
atomic decomposition.

\begin{theorem}\label{t3.1}
Let $\fai$ be as in Definition \ref{d2.2}. Then $f\in
T_{\fai}(\rr^{n+1}_+)$ if and only if there exist $\{\lz_j\}_j\subset\cc$ and a sequence
$\{a_j\}_j$ of $(\fai,\fz)$-atoms such that, for almost every $(x,t)\in\rr^{n+1}_+$,
\begin{equation}\label{3.1}
f(x,t)=\sum_{j}\lz_ja_j(x,t)
\end{equation}
and
\begin{equation}\label{3.x1}
\sum_j\fai\lf(B_j,
|\lz_j|\|\chi_{B_j}\|^{-1}_{L^\fai(\rn)}\r)<\fz,
\end{equation}
where, for each $j$, $\widehat{B_j}$ appears in the support of $a_j$.
Moreover, there exists a positive constant $C$ such that, for all
$f\in T_{\fai}(\rr^{n+1}_+)$,
\begin{eqnarray}\label{3.2}
\qquad\ \blz(\{\lz_j a_j\}_j):=\inf\lf\{\lz\in(0,\fz):\
\sum_j\fai\lf(B_j,\frac{|\lz_j|}
{\lz\|\chi_{B_j}\|_{L^\fai(\rn)}}\r)\le1\r\}\sim
\|f\|_{T_{\fai}(\rr^{n+1}_+)},
\end{eqnarray}
where the implicit constants are independent of $f$.
\end{theorem}

The proof of Theorem \ref{t3.1} is similar to that of
\cite[Theorem 1(a)]{cms85} or \cite[Theorem 3, p.\,64]{st93}.
To give the details, we need some known facts as follows.

Let $F$ be a closed subset of $\rn$ and $O:=F^\complement$. Assume that $|O|<\fz$.
For any fixed $\gz\in(0,1)$, we say that $x\in\rn$ has the \emph{global $\gz$-density}
with respect to $F$ if, for all $r\in(0,\fz)$,
$$\frac{|B(x,r)\cap F|}{|B(x,r)|}\ge\gz.$$
Denote by $F_\gz^{\ast}$ the set of all such $x$. It is easy to prove
that $F_\gz^{\ast}$ with $\gz\in(0,1)$ is a closed subset of $F$. Let $\gz\in(0,1)$
and $O_\gz^{\ast}:=(F_\gz^{\ast})^\complement$. Then $O_\gz^{\ast}$ is open and
$O\subset O_\gz^{\ast}$. Indeed, from the definition of $O_\gz^{\ast}$, we deduce
that
\begin{equation}\label{3.1x}
O_\gz^\ast=\{x\in\rn:\
\cm(\chi_O)(x)>1-\gz\},
\end{equation}
which, together with the fact that $\cm$ is of weak type $(1,1)$, further implies that
there exists a positive constant $C(\gz)$, depending on $\gz$, such
that $|O_\gz^\ast|\le C(\gz)|O|$.

The following lemma is just \cite[Lemma 3.1]{jy10}.

\begin{lemma}\label{l3.1}
There exist positive constants
$\gz\in(0, 1)$ and $C(\gz)$ such that, for any closed subset $F$
of $\rn$ whose complement has finite measure,
and any nonnegative measurable function $H$ on $\rr^{n+1}_+$, it holds that
$$\int_{\ccr(F^\ast_\gz)}H(y,t)t^n\,dy\,dt\le
C(\gz)\int_F\lf\{\int_{\bgz(x)}H(y,t)\,dy\,dt\r\}\,dx,$$
where $F^\ast_\gz$ denotes the set of points in $\rn$ with the global
$\gz$-density with respect to $F$.
\end{lemma}

Moreover, we also need the following lemma, whose proof
is similar to that of \cite[Lemma 5.4]{k}, the  details being omitted.

\begin{lemma}\label{l3.2}
Let $\fai$ be as in Definition \ref{d2.2}, $f\in T_{\fai}(\rr^{n+1}_+)$, $k\in\zz$ and
$$\boz_k:=\lf\{x\in\rn:\ \ca(f)(x)>2^k\r\}.$$
 Then there exists a positive constant $C$ such that, for all $\lz\in(0,\fz)$,
$$\sum_{k\in\zz}\fai\lf(\boz_k,\frac{2^k}{\lz}\r)\le
C\int_{\rn}\fai\lf(x,\frac{\ca(f)(x)}{\lz}\r)\,dx.
$$
\end{lemma}

Now we prove Theorem \ref{t3.1} by using Lemmas \ref{l3.1} and \ref{l3.2}.

\begin{proof}[Proof of Theorem \ref{t3.1}]
Assume first that there exist $\{\lz_j\}_j\subset\cc$ and a sequence $\{a_j\}$
of $(\fai,\fz)$-atoms such that \eqref{3.1} and \eqref{3.x1} hold true.
By Minkowski's inequality for integrals, the definition of $\ca(f)$, and
Lemmas \ref{l3.1a} and \ref{l2.1}(i), we conclude that, for all $\lz\in(0,\fz)$,
\begin{eqnarray*}
\int_{\rn}\fai\lf(x,\frac{\ca(f)(x)}{\lz}\r)\,dx
&&\ls\sum_j\int_{\rn}\fai\lf(x,\frac{|\lz_j|\ca(a_j)(x)}{\lz}\r)\,dx\\
&&\ls\sum_j\fai\lf(B_j,\frac{|\lz_j|}{\lz\|\chi_{B_j}\|_{L^\fai(\rn)}}\r),
\end{eqnarray*}
which, together with \eqref{3.x1} and the definitions of
$\blz(\{\lz_j a_j\}_j)$ and $\|f\|_{T_{\fai}(\rr^{n+1}_+)}$, implies that
$f\in T_{\fai}(\rr^{n+1}_+)$ and $\|f\|_{T_{\fai}(\rr^{n+1}_+)}\ls\blz(\{\lz_j a_j\}_j)$.

Conversely, let $f\in T_{\fai}(\rr^{n+1}_+)$. For any $k\in\zz$, let
$$O_k:=\lf\{x\in\rn:\ \ca(f)(x)>2^k\r\}$$
and $F_k:=O_k^\complement$. Since $f\in T_{\fai}(\rr^{n+1}_+)$, for each
$k$, $O_k$ is an open set of $\rn$ and $|O_k|<\fz$.

Let $\gz\in(0,1)$ be as in Lemma \ref{l3.1}. In what follows,
we simplify notations and write $(F_{k})_\gz^\ast$ and $(O_{k})_\gz^\ast$ as
$F_k^\ast$ and $O_k^\ast$, respectively. We claim that $\supp f\subset(\cup_{k\in\zz}
\widehat{O^{\ast}_k}\cup E)$, where $E\subset\rr^{n+1}_+$ satisfies that
$\int_E\frac{dy\,dt}{t}=0$. To see this, let $(x,t)\in\rr^{n+1}_+$ be a
Lebesgue point of $f$ and $(x,t)\not\in
\cup_{k\in\zz}\widehat{O^{\ast}_k}$. Then, by $(x,t)\not\in
\cup_{k\in\zz}\widehat{O^{\ast}_k}$, we know that there exists a sequence
$\{y_k\}_{k\in\zz}$ of points such that $\{y_k\}_{k\in\zz}\subset B(x,t)$
and, for each $k$, $y_k\not\in O^{\ast}_k$, which, together with \eqref{3.1x},
implies that, for each $k\in\zz$,
$\cm(\chi_{O_k})(y_k)\le 1-\gz$. From
this, we further deduce that $|B(x,t)\cap O_k|\le (1-\gz)|B(x,t)|$ and hence
$$|B(x,t)\cap\{z\in\rn:\ \ca(f)(z)\le2^k\}|\ge\gz|B(x,t)|.$$
Letting $k\to-\fz$, we then see that $|B(x,t)\cap\{z\in\rn:\
\ca(f)(z)=0\}|\ge\gz|B(x,t)|$.
Therefore, since $\gz\in(0,1)$, it follows that there exists
$y\in B(x,t)$ such that $\ca(f)(y)=0$. By this and the definition of $\ca(f)$, we see that
$f=0$ almost everywhere in $\bgz(y)$,
which, together with Lebesgue's differentiation theorem, implies that
$f(x,t)=0$. From this and the fact that almost every $(x,t)\in\rr^{n+1}_+$ is
a Lebesgue point of $f$, we infer that the claim holds true.

Recall that $O^\ast_k$, for each $k\in\zz$, is open. Moreover, for each
$k\in\zz$, considering a Whitney decomposition of the set $O^\ast_k$, we obtain
a set $I_k$ of indices and a family $\{Q_{k,\,j}\}_{j\in I_k}$ of closed cubes
with disjoint interiors such that

(i) $\cup_{j\in I_k}Q_{k,j}=O^\ast_k$ and, if $i\neq j$, then $\mathring Q_{k,j}
\cap\mathring Q_{k,i}=\emptyset$, where $\mathring E$ denotes the \emph{interior} of the set $E$;

(ii) $\sqrt{n}\ell(Q_{k,j})\le\dist(Q_{k,j}, (O^\ast_k)^\complement)\le4
\sqrt{n}\ell(Q_{k,j})$, where $\ell(Q_{k,j})$ denotes the
\emph{side-length} of $Q_{k,j}$ and $\dist(Q_{k,j}, (O^\ast_k)^\complement)
:=\inf\{d(u,w):\ u\in Q_{k,j},\ w\in (O^\ast_k)^\complement\}$.

Then for each $j\in I_k$, we let $B_{k,j}$ be the \emph{ball with the
center same as $Q_{k,j}$ and with the radius $\frac{11}{2}\sqrt{n}$-times
$\ell(Q_{k,j})$}. Let $A_{k,j}:=\widehat{B_{k,j}}\cap(Q_{k,j}
\times(0,\fz))\cap(\widehat{O^\ast_k}\setminus \widehat{O^\ast_{k+1}})$,
$$a_{k,j}:=2^{-k}\|\chi_{B_{k,j}}\|^{-1}_{L^\fai(\rn)}
f\chi_{A_{k,j}}$$
and $\lz_{k,j}:=2^{k}\|\chi_{B_{k,j}}\|_{L^\fai(\rn)}$. Notice that $\{(Q_{k,j}
\times(0,\fz))\cap(\widehat{O^\ast_k}\setminus \widehat{O^\ast_{k+1}})\}
\subset\widehat{B_{k,j}}$. From this, we deduce that
\begin{equation}\label{3.x2}
f=\sum_{k\in\zz}\sum_{j\in I_k}
\lz_{k,j}a_{k,j}
\end{equation}
almost everywhere on $\rr^{n+1}_+$.

We first show that, for each $k\in\zz$ and $j\in I_k$, $a_{k,j}$ is
a $(\fai,\,\fz)$-atom supported in $\widehat{B_{k,j}}$. Let $p\in(1,\fz)$, $p'$
be its \emph{conjugate index}, and
$h\in T^{p'}_2(\rr^{n+1}_+)$ with $\|h\|_{T^{p'}_2(\rr^{n+1}_+)}\le1$.
Since $A_{k,j}\subset(\widehat{O^\ast_{k+1}})^\complement=F^\ast_{k+1}$,
by Lemma \ref{l3.1} and H\"older's inequality, we see that
\begin{eqnarray*}
|\langle a_{k,j},h\rangle|&&:=\lf|\int_{\rr^{n+1}_+}a_{k,j}(y,t)
\chi_{A_{k,j}}(y,t)h(y,t)\frac{dy\,dt}{t}\r|\\ \nonumber
&&\ls\int_{F_{k+1}}\int_{\bgz(x)}|a_{k,j}(y,t)h(y,t)|
\frac{dy\,dt}{t^{n+1}}\,dx
\ls\int_{(O_{k+1})^\complement}\ca(a_{k,j})(x)\ca(h)(x)\,dx\\ \nonumber
&&\ls2^{-k}\|\chi_{B_{k,j}}\|^{-1}_{L^\fai(\rn)}
\lf\{\int_{B_{k,j}\cap(O_{k+1})^\complement}
[\ca(f)(x)]^p\,dx\r\}^{1/p}\|h\|_{T^{p'}_2(\rr^{n+1}_+)}\\ \nonumber
&&\ls|B_{k,j}|^{1/p}\|\chi_{B_{k,j}}\|^{-1}_{L^\fai(\rn)},
\end{eqnarray*}
which, together with $(T^p_2(\rr^{n+1}_+))^\ast=T^{p'}_2(\rr^{n+1}_+)$
(see \cite[Theorem 2]{cms85}), where $(T^p_2(\rr^{n+1}_+))^\ast$ denotes
the \emph{dual space} of $T^p_2(\rr^{n+1}_+)$ and $1/p+1/p'=1$, implies that
$\|a_{k,j}\|_{T_2^p(\rr^{n+1}_+)}\ls|B_{k,j}|^{1/p}
\|\chi_{B_{k,j}}\|^{-1}_{L^\fai(\rn)}$. Thus, $a_{k,j}$ is a
$(\fai,p)$-atom supported in $\widehat{B_{k,j}}$ up to a harmless constant
for all $p\in(1,\fz)$ and hence a $(\fai,\fz)$-atom up to a harmless constant.

Since $\fai\in\aa_\fz(\rn)$, by Lemma \ref{l2.4}(iv), we know that there
exists $p_0\in(q(\fai),\fz)$
such that $\fai\in\aa_{p_0}(\rn)$. From this and Lemma \ref{l2.4}(v), it
follows that, for any $k\in\zz$ and $t\in(0,\fz)$,
\begin{eqnarray*}
\fai\lf(O^\ast_k,t\r)&&\ls\frac{1}{(1-\gz)^{p_0}}\int_{O^\ast_k}
\lf[\cm(\chi_{O_k})(x)\r]^{p_0}\fai(x,t)\,dx\\
&&\ls
\frac{1}{(1-\gz)^{p_0}}\int_{\rn}
\lf[\chi_{O_k}(x)\r]^{p_0}\fai(x,t)\,dx\sim\fai\lf(O_k,t\r),
\end{eqnarray*}
which, together with Lemmas \ref{l2.4}(vi) and
\ref{l3.2}, and the property (i) of $\{Q_{k,j}\}_{k\in\zz,\,j\in I_k}$,
implies that, for all $\lz\in(0,\fz)$,
\begin{eqnarray}\label{3.x3}
&&\sum_{k\in\zz}\sum_{j\in I_k}\fai\lf(
B_{k,j},\frac{|\lz_{k,j}|}{\lz\|\chi_{
B_{k,j}}\|_{L^\fai(\rn)}}\r)\\ \nonumber
&&\hs\ls\sum_{k\in\zz}\sum_{j\in
I_k}\fai\lf(B_{k,j},\frac{2^k}{\lz}\r)\ls\sum_{k\in\zz}\sum_{j\in
I_k}\fai\lf(Q_{k,j},\frac{2^k}{\lz}\r)\ls\sum_{k\in\zz}
\fai\lf(O^\ast_k,\frac{2^k}{\lz}\r)\\ \nonumber
&&\hs\ls\sum_{k\in\zz}
\fai\lf(O_k,\frac{2^k}{\lz}\r)
\ls\int_{\rn}\fai\lf(x,\frac{\ca(f)(x)}{\lz}\r)\,dx.
\end{eqnarray}
By this, we conclude that
$\blz(\{\lz_{k,j}a_{k,j}\}_{k\in\zz,\,j})\ls\|f\|_{T_{\fai}(\rr^{n+1}_+)}$, which
completes the proof of Theorem \ref{t3.1}.
\end{proof}

\begin{remark}\label{r3.x1}
Let $\{a_{k,j}\}_{k\in\zz,\,j\in I_k}$ be as in \eqref{3.x2}.
Then $\{\supp(a_{k,j})\}_{k\in\zz,\,
j\in I_k}$ have pairwise disjoint interior and
$\int_{\supp f\setminus\cup_{k\in\zz,\,
j\in I_k}\supp(a_{k,j})}\frac {dy\,dt}t=0$. Indeed, let
$\{A_{k,j}\}_{k\in\zz,\,j\in I_k}$, $\{Q_{k,j}\}_{k\in\zz,\,j\in I_k}$
and $\{\widehat{O^\ast_{k}}\}_{k\in\zz}$ be as in the proof of Theorem \ref{t3.1}.
Then by the definition of the set $A_{k,j}$, the fact that $\mathring Q_{k,j}\cap
\mathring Q_{k,j_1}=\emptyset$ for any $k\in\zz$ and $j,\,j_1\in I_k$
with $j\neq j_1$, and the observation that $(\widehat{O^\ast_k}\setminus \widehat{O^\ast_{k+1}})\cap
(\widehat{O^\ast_{k_1}}\setminus \widehat{O^\ast_{k_1+1}})=\emptyset$ for any $k,\,k_1\in\zz$
and $k\neq k_1$, we conclude that the collection of sets, $\{A_{k,j}\}_{k\in\zz,\,j\in I_k}$, are pairwise
disjoint, up to sets of measure zero. From this and the definitions of
$\{a_{k,j}\}_{k\in\zz,\,j\in I_k}$, we infer that this claim holds true.
\end{remark}

\begin{corollary}\label{c3.1}
Let $p\in(0,\fz)$ and $\fai$ be as in Definition \ref{d2.2}. If $f\in
T_{\fai}(\rr^{n+1}_+)\cap T^p_2(\rr^{n+1}_+) $, then the
decomposition \eqref{3.1} also holds in both $T_{\fai}(\rr^{n+1}_+)$ and
$T^p_2(\rr^{n+1}_+)$.
\end{corollary}

\begin{proof}
Let $f\in T_{\fai}(\rr^{n+1}_+)\cap T^p_2(\rr^{n+1}_+)$. We first show that
\eqref{3.1} holds in $T_{\fai}(\rr^{n+1}_+)$.
Assume that, for each $k$ and $j$, $\lz_{k,j}$, $a_{k,j}$ and
$B_{k,j}$ are as in the proof of Theorem \ref{t3.1}. By Lemma \ref{l3.1a},
we see that
\begin{equation}\label{3.3}
\int_{\rn}\fai\lf(x,\ca(\lz_{k,j}a_{k,j})(x)\r)\,dx
\ls\fai\lf(B_{k,j},\frac{|\lz_{k,j}|}
{\|\chi_{B_{k,j}}\|_{L^\fai(\rn)}}\r),
\end{equation}
Moreover, it was proved in Theorem \ref{t3.1} (see \eqref{3.x3}) that
$$\sum_{k\in\zz}\sum_{j\in I_k}\fai\lf(B_{k,j},
\frac{|\lz_{k,j}|}{\|\chi_{B_{k,j}}
\|_{L^\fai(\rn)}}\r)\ls\int_{\rn}\fai(x,\ca(f)(x))\,dx<\fz.
$$
By this, \eqref{3.1}, Lemma \ref{l2.1}(i) and \eqref{3.3}, we conclude that
\begin{eqnarray*}
&&\int_{\rn}\fai\lf(x,
\ca\lf(f-\sum_{|k|+j<N}\lz_{k,j}a_{k,j}\r)(x)\r)\,dx\\
&&\hs\ls\sum_{|k|+j\ge N}\int_{\rn}\fai\lf(x,
\ca(\lz_{k,j}a_{k,j})(x)\r)\,dx\ls\sum_{|k|+j\ge
N}\fai\lf(B_{k,j}, \frac{|\lz_{k,j}|}{\|\chi_{B_{k,j}}
\|_{L^\fai(\rn)}}\r)\to0,
\end{eqnarray*}
as $N\to\fz$. Therefore, \eqref{3.1} holds true in $T_{\fai}(\rr^{n+1}_+)$.

Moreover, similar to the proof of
\cite[Proposition 3.1]{jy10}, we
know that \eqref{3.1} also holds true in $T^p_2(\rr^{n+1}_+)$, which
completes the proof of Corollary \ref{c3.1}.
\end{proof}

In what follows, let $T^c_{\fai}(\rr^{n+1}_+)$ and $T^{p,\,c}_2(\rr^{n+1}_+)$ with
$p\in(0,\fz)$ denote, respectively, the sets of all functions
in $T_{\fai}(\rr^{n+1}_+)$ and $T^p_2(\rr^{n+1}_+)$ with compact support.

\begin{proposition}\label{p3.1}
Let $\fai$ be as in Definition \ref{d2.2}. Then
$T^c_{\fai}(\rr^{n+1}_+)\subset T^{2,\,c}_2(\rr^{n+1}_+)$ as sets.
\end{proposition}

\begin{proof}
It is well known that, for all $p\in(0,\fz)$, $T^{p,\,c}_2(\rr^{n+1}_+)\subset
T^{2,\,c}_2(\rr^{n+1}_+)$ as sets (see, for example, \cite[p.\,306, (1.3)]{cms85}).
Thus, to prove $T^c_{\fai}(\rr^{n+1}_+)\subset T^{2,\,c}_2(\rr^{n+1}_+)$,
it suffices to show that $T^c_{\fai}(\rr^{n+1}_+)\subset T^{p,\,c}_2(\rr^{n+1}_+)$
for some $p\in(0,\fz)$. Suppose that $f\in T^c_{\fai}(\rr^{n+1}_+)$ and $\supp (f)\subset K$,
where $K$ is a compact set in $\rr^{n+1}_+$. Let
$B$ be a ball in $\rn$ such that $K\subset\widehat{B}$. Then
$\supp(\ca(f))\subset\widehat{B}$. Let $p_0\in(0,i(\fai))$ and $q_0\in(q(\fai),\fz)$.
Then $\fai$ is of uniformly lower type $p_0$ and $\fai\in\aa_{q_0}(\rn)$. From this,
H\"older's inequality, \eqref{2.2} and the uniformly
lower type $p_0$ property of $\fai$, we deduce that
\begin{eqnarray*}
&&\int_{\rn}[\ca(f)(x)]^{p_0/q_0}\,dx\\
&&\hs\le\lf\{\int_B[\ca(f)(x)]^{p_0}\fai(x,1)\,dx\r\}^{1/q_0}
\lf\{\int_B[\fai(x,1)]^{-q'_0/q_0}\,dx\r\}^{1/q'_0}\\
&&\hs\ls\frac{|B|}{[\fai(B,1)]^{1/q_0}}\lf\{\int_{\{x\in B:\
\ca(f)(x)\le1\}}[\ca(f)(x)]^{p_0}\fai(x,1)\,dx
+\int_{\{x\in B:\ \ca(f)(x)>1\}}\cdots\r\}^{1/q_0}\\
&&\hs\ls\frac{|B|}{[\fai(B,1)]^{1/q_0}}\lf\{\fai(B,1)
+\int_B\fai(x,\ca(f)(x))\,dx\r\}^{1/q_0}<\fz,
\end{eqnarray*}
where $1/q_0+1/q_0'=1$, which implies that $f\in
T^{p_0/q_0,\,c}_2(\rr^{n+1}_+)\subset T^{2,\,c}_2(\rr^{n+1}_+)$. This finishes
the proof of Proposition \ref{p3.1}.
\end{proof}

\section{Lusin area function and molecular characterizations \\of $H_\fai(\rn)$\label{s4}}

\hskip\parindent In this section, we first recall the Musielak-Orlicz Hardy
space $H_\fai(\rn)$ introduced by Ky \cite{k}. Then we establish two
equivalent characterizations of $H_\fai(\rn)$ in terms of the molecule
and the Lusin area function. We begin with some notion and notations.

In what follows, we denote by $\cs(\rn)$ the space of all Schwartz functions and
by $\cs'(\rn)$ its \emph{dual space} (namely, the space of all \emph{tempered distributions}).
For $m\in\nn$, define
$$\cs_m(\rn):=\lf\{\phi\in\cs(\rn):\ \sup_{x\in\rn}\sup_{
\bz\in\zz^n_+,\,|\bz|\le m+1}(1+|x|)^{(m+2)(n+1)}|\partial^\bz_x\phi(x)|\le1\r\}.
$$
Then for all $f\in\cs'(\rn)$, the \emph{non-tangential grand maximal function} $f^\ast_m$
of $f$ is defined by setting, for all $x\in\rn$,
$$f^\ast_m(x):=\sup_{\phi\in\cs_m(\rn)}\sup_{|y-x|<t,\,t\in(0,\fz)}|f\ast\phi_t(y)|,
$$
where, for all $t\in(0,\fz)$, $\phi_t(\cdot):=t^{-n}\phi(\frac{\cdot}{t})$.
When $m(\fai):=\lfz n[q(\fai)/i(\fai)-1]\rfz$, where $q(\fai)$ and $i(\fai)$ are,
respectively, as in \eqref{2.3} and \eqref{2.1}, we denote $f^\ast_{m(\fai)}$ simply
by $f^\ast$.

Now we recall the definition of the Musielak-Orlicz Hardy $H_\fai(\rn)$ introduced
by Ky \cite{k} as follows.

\begin{definition}\label{d4.1}
Let $\fai$ be as in Definition \ref{d2.2}. The \emph{Musielak-Orlicz Hardy space
$H_\fai(\rn)$} is defined as the space of all $f\in\cs'(\rn)$ such that
$f^\ast\in L^\fai(\rn)$
with the \emph{quasi-norm}
$\|f\|_{H_\fai(\rn)}:=\|f^\ast\|_{L^\fai(\rn)}$.
\end{definition}

\begin{definition}\label{d4.2}
Let $\fai$ be as in Definition \ref{d2.2}. Assume that
$\phi\in\cs(\rn)$ is a radial real-valued function satisfying that,
for all $\gz\in\zz_+^n$ and $|\gz|\le s$,
where $s\in\zz_+$ and $s\ge\lfz n[q(\fai)/i(\fai)-1]\rfz$,
\begin{equation}\label{4.1}
\int_{\rn}\phi(x)x^{\gz}\,dx=0
\end{equation}
and, for all $\xi\in\rn\setminus\{0\}$,
\begin{equation}\label{4.x1}
\int_0^\fz|\widehat{\phi}(t\xi)|^2\frac{dt}{t}=1,
\end{equation}
 where $\widehat{\phi}$ denotes the \emph{Fourier transform} of
$\phi$.

Then for all $f\in\cs'(\rn)$ and $x\in\rn$, define
$$S(f)(x):=\lf\{\int_{\bgz(x)}|\phi_t\ast f(y)|^2\frac{dy\,dt}{t^{n+1}}\r\}^{1/2}.$$
\end{definition}

It is known that the Lusin area function $S$ is bounded on $L^p(\rn)$ for all
$p\in(1,\fz)$ (see, for example, \cite[Theorem 7.8]{fs}).

Now we introduce the Musielak-Orlicz Hardy space $H_{\fai,S}(\rn)$ via the
Lusin area function as follows.

\begin{definition}\label{d4.3}
Let $\fai$ be as in Definition \ref{d2.2}. The \emph{Musielak-Orlicz Hardy space
$H_{\fai,S}(\rn)$} is defined as the space of all $f\in\cs'(\rn)$ such that
$S(f)\in L^\fai(\rn)$
with the \emph{quasi-norm}
$$\|f\|_{H_{\fai,S}(\rn)}:=\|S(f)\|_{L^\fai(\rn)}:=
\inf\lf\{\lz\in(0,\fz):\
\int_{\rn}\fai\lf(x,\frac{S(f)(x)}{\lz}\r)\,dx\le1\r\}.
$$
\end{definition}

To introduce the molecular Musielak-Orlicz Hardy space, we first
introduce the notion of the molecule associated with the growth function $\fai$.

\begin{definition}\label{d4.4}
Let $\fai$ be as in Definition \ref{d2.2}, $q\in(1,\fz)$, $s\in\zz_+$ and
$\varepsilon\in(0,\fz)$. We say that a function $\az\in L^q(\rn)$ is a
\emph{$(\fai,q,s,\uc)$-molecule} associated with the ball $B$ if

(i) for each $j\in\zz_+$, $\|\az\|_{L^q(U_j(B))}\le
2^{-j\uc}|2^j B|^{1/q}\|\chi_{B}\|_{L^\fai(\rn)}^{-1}$,
where $U_0(B):=B$ and $U_j(B):=2^j B\setminus2^{j-1}B$ for $j\in\nn$;

(ii) for all $\bz\in\zz_+^n$ with $|\bz|\le s$,
$\int_{\rn}\az(x)x^{\bz}\,dx=0$.\\
Moreover, if $\az$ is a $(\fai,q,s,\uc)$-molecule for all $q\in(1,\fz)$,
we then say that $\az$ is a \emph{$(\fai,\fz,s,\uc)$-molecule}.
\end{definition}

\begin{definition}\label{d4.5}
Let $\fai$ be as in Definition \ref{d2.2}, $p,\,q\in(1,\fz)$, $s\in\zz_+$
and $\uc\in(0,\fz)$. The \emph{molecular Musielak-Orlicz Hardy space}
$H^{q,s,\uc}_{\fai,\mathrm{mol}}(\rn)$ is defined as the space of all
$f\in\cs'(\rn)$ satisfying that $f=\sum_j\lz_j\az_j$ in $\cs'(\rn)$,
where $\{\lz_j\}\subset\cc$, $\{\az_j\}_j$ is a sequence of
$(\fai,\,q,\,s,\,\uc)$-molecules and
$\sum_j\fai(B_j,\|\chi_{B_j}\|^{-1}_{L^\fai(\rn)})<\fz$,
where, for each $j$, the molecule $\az_j$ is associated with the ball $B_j$.
Moreover, define
$\|f\|_{H^{q,s,\uc}_{\fai,\mathrm{mol}}(\rn)}:
=\inf\{\Lambda(\{\lz_j\az_j\}_{j\in{\nn}})\}$,
where the infimum is taken over all decompositions of $f$ as above and
$$\Lambda
\lf(\lf\{\lz_j\az_j\r\}_{j\in{\nn}}\r):=
\inf\lf\{\lz\in(0,\fz):\ \sum_{j\in\nn}\fai\lf(B_j,\frac{|\lz_j|}{
\lz\|\chi_{B_j}\|_{L^\fai(\rn)}}\r)\le1\r\}.$$
\end{definition}

\begin{definition}\label{d4.6}
Let $\phi$ be as in Definition \ref{d4.2}. For all $f\in T^{p,\,c}_2(\rr^{n+1}_+)$
with $p\in(1,\fz)$ and $x\in\rn$, define
\begin{equation}\label{4.2}
\pi_{\phi}(f)(x):=\int_0^{\fz}\lf(f(\cdot,t)\ast\phi_t\r)(x)
\,\frac{dt}{t}.
\end{equation}
\end{definition}

It was proved in \cite{cms85} that
$\pi_{\phi}(f)\in L^2(\rn)$ for such an $f$. Moreover, we have the following
properties for the operator $\pi_{\phi}$.

\begin{proposition}\label{p4.1}
Let  $\pi_{\phi}$ be as in \eqref{4.2} and $\fai$
as in Definition \ref{d2.2}. Then

{\rm(i)} the operator $\pi_{\phi}$, initially defined on the
space $T^{p,\,c}_2(\rr^{n+1}_+)$ with $p\in(1,\fz)$, extends to a bounded linear
operator from $T^p_2(\rr^{n+1}_+)$ to $L^p(\rn)$;

{\rm(ii)} the operator $\pi_{\phi}$, initially defined on the space
$T^c_{\fai}(\rr^{n+1}_+)$, extends to a bounded linear operator
from $T_{\fai}(\rr^{n+1}_+)$ to $H_{\fai,S}(\rn)$.
\end{proposition}

To prove Proposition \ref{p4.1}(ii), we need the following lemma.

\begin{lemma}\label{l4.1a}
Let $\pi_\phi$ and $s$ be respectively as in \eqref{4.2} and Definition \ref{4.2}.
Then for any $\epz\in(0,\fz)$ and any $(\fai,\fz)$-atom $a$ with $\supp (a)
\subset\widehat{B}$ and $B$ being a ball,
$\pi_\phi(a)$ is a harmless constant multiple of a $(\fai,\fz,s,n+\epz)$-molecule
associated with the ball $B$.
\end{lemma}

\begin{proof}
Let $a$ be a $(\fai,\fz)$-atom supported in
the ball $B:=B(x_B, r_B)$ and $q\in(1,\fz)$. Since,  for any $q\in(1,2)$ and
$\epz\in(0,\fz)$,
each $(\fai,2,s,n+\epz)$-molecule is also a $(\fai,q,s,n+\epz)$-molecule,
to prove this lemma, it suffices to show that, for any $\epz\in(0,\fz)$,
$\az:=\pi_{\phi}(a)$ is a harmless constant multiple of a $(\fai,q,s,
n+\epz)$-molecule associated with $B$ for $q\in[2,\fz)$.

Let $q\in[2,\fz)$. When $j\in\{0,\,\ldots,\,4\}$,
by the fact that $\pi_{\phi}$ is bounded from $T^q_2(\rr^{n+1}_+)$ to $L^q(\rn)$
(see Proposition \ref{p4.1}(i)), we know that
\begin{eqnarray}\label{4.4}
\|\az\|_{L^q(U_j(B))}=\|\pi_{\phi}(a)\|_{L^q(U_j(B))}
\ls\|a\|_{T^q_2(\rr^{n+1}_+)}\ls|B|^{1/q}
\|\chi_B\|^{-1}_{L^\fai(\rn)}.
\end{eqnarray}

When $j\in\nn$ and $j\ge4$, take $h\in L^{q'}(\rn)$ satisfying
$\|h\|_{L^{q'}(\rn)}\le1$
and $\supp(h)\subset U_j(B)$. Then from H\"older's inequality and
$q'\in(1,2]$, we infer that
\begin{eqnarray}\label{4.5}
\qquad|\langle \pi_{\phi}(a),h\rangle|
&&\le\int_B\int_0^{r_B}|a(x,t)||\phi_t\ast h(x)|\,dx\,\frac{dt}{t}\\ \nonumber
&&\ls\|\ca(a)\|_{L^q(\rn)}\lf\|\ca\lf(\chi_{\widehat{B}}\phi_t\ast
h\r)\r\|_{L^{q'}(\rn)}\\ \nonumber
&&\ls\|a\|_{T^q_2(\rr^{n+1}_+)}
|B|^{1/q'-1/2}\lf\{\int_{\widehat{B}}|\phi_t\ast h(x)|^2\,\frac{dx\,dt}{t}\r\}^{1/2}.
\end{eqnarray}

Let $\epz\in(0,\fz)$. Then by this, $\phi\in\cs(\rn)$,
H\"older's inequality and the fact that, for any $x\in B$ and
$y\in U_j(B)$, $|x-y|\gs2^{j-1}r_B$, we conclude that, for all $x\in B$,
\begin{eqnarray*}
|\phi_t\ast h(x)|
&&\ls\int_{\rn}\frac{t^{\epz}}{(t+|x-y|)^{n+\epz}}
|h(y)|\,dy\\
&&\ls\frac{t^{\epz}}{(2^{j}r_B)^{n+\epz}}
\|h\|_{L^{q'}(\rn)}|2^jB|^{1/q}\ls
\frac{t^{\epz}}{(2^{j}r_B)^{n/q'+\epz}},
\end{eqnarray*}
which, together with \eqref{4.5}, implies that
$$|\langle \pi_{\phi}(a),h\rangle|\ls2^{-j(n+\epz)}|2^jB|^{1/q}
\|\chi_B\|^{-1}_{L^\fai(\rn)}.$$
From this and the choice of $h$, we deduce that, for each
$j\in\nn$ and $j\ge4$,
\begin{eqnarray}\label{4.6}
\|\az_j\|_{L^{q}(U_j(B))}=
 \|\pi_\phi(a)\|_{L^{q}(U_j(B))}
 \ls2^{-j(n+\epz)}|2^jB|^{1/q}
\|\chi_B\|^{-1}_{L^\fai(\rn)}.
\end{eqnarray}
 Moreover, by \eqref{4.1}, we know that,
 for all $\gz\in\zz_+^n$ with $|\gz|\le s$,
$$\int_{\rn}\pi_\phi(a)(x)x^\gz\,dx=\int_0^\fz\lf\{\int_\rn
\int_\rn\phi_t(x-y)x^\gz\,dx\r\}a(y,t)\,\frac{dy\,dt}{t}=0,$$
which, together with \eqref{4.4} and \eqref{4.6}, implies that $\az$
is a  harmless constant multiple of a $(\fai,q,s,n+\epz)$-molecule
associated with $B$. This finishes the proof of Lemma \ref{l4.1a}.
\end{proof}

Now we prove Proposition \ref{p4.1} by using Proposition \ref{p3.1}, Corollary
\ref{c3.1} and Lemma \ref{l4.1a}.

\begin{proof}[Proof of Proposition \ref{p4.1}]
The conclusion (i) is just \cite[Theorem 6(1)]{cms85}.

Now we prove (ii). Let
$f\in T^c_{\fai}(\rr^{n+1}_+)$. Then by Proposition \ref{p3.1},
Corollary \ref{c3.1} and (i), we know that
$$\pi_{\phi}(f)=\sum_j\lz_j\pi_{\phi}(a_j)
=:\sum_j\lz_j\az_j\ \text{in}\  L^2(\rn),$$
where the sequences $\{\lz_j\}_j$ and $\{a_j\}_j$ satisfy
\eqref{3.1} and \eqref{3.2}. Recall that, for each $j$,
$\supp (a_j)\subset\widehat{B_j}$ and $B_j$ is a ball of $\rn$. Moreover,
from Minkowski's inequality for integrals,
we deduce that, for all $x\in\rn$,
$S(\pi_{\phi}(f))(x)\le\sum_j|\lz_j|S(\az_j)(x)$. This,
combined with Lemma \ref{l2.1}(i), yields that
\begin{equation}\label{4.3}
\int_{\rn}\fai(x,S(\pi_{\phi}(f))(x))\,dx\ls
\sum_j\int_{\rn}\fai(x,|\lz_j|S(\az_j)(x))\,dx.
\end{equation}

By Lemma \ref{l4.1a} with $\epz\in(n[q(\fai)/i(\fai)-1],\fz)$, we see that,
for each $j$, $\az_j:=\pi_{\phi}(a_j)$ is a harmless constant multiple of a
$(\fai,\fz,s,n+\epz)$-molecule associated with the ball $B_j$ for each $j$,
where $s$ is as in Definition \ref{d4.2}.

By $\epz>n[q(\fai)/i(\fai)-1]$ and $s\ge\lfz n[q(\fai)/i(\fai)-1]\rfz$,
we know that there exist $p_0\in(0,i(\fai))$ and $q_0\in(q(\fai),\fz)$
such that $\epz>n(q_0/p_0-1)$ and $s+1>n(q_0/p_0-1)$. Then $\fai\in\aa_{q_0}(\rn)$
and $\fai$ is of uniformly lower type $p_0$. Let $\wz\epz:=n+\epz$ and
$q\in[2,\fz)$ satisfying $q'<r(\fai)$. Then $\fai\in\rh_{q'}(\rn)$.
We now claim that, for any $\lz\in\cc$
and $(\fai,q,s,\wz\epz)$-molecule $\az$ associated with the ball $B\subset\rn$,
it holds that
\begin{equation}\label{4.7}
\int_{\rn}\fai(x,S(\lz
\az)(x))\,dx\ls\fai\lf(B,\frac{|\lz|}{\|\chi_B\|_{L^\fai(\rn)}}\r).
\end{equation}

Assuming that \eqref{4.7} holds for a moment, then
from \eqref{4.7}, the facts that, for all $\lz\in(0,\fz)$,
$S(\pi_{\phi}(f/\lz))=S(\pi_{\phi}(f))/\lz$, $\pi_{\phi}(f/\lz)=\sum_j\lz_j\az_j/\lz$
and $S(\pi_{\phi}(f))\le\sum_j|\lz_j|S(\az_j)$,
it follows that, for all $\lz\in(0,\fz)$,
$$\int_{\rn}\fai\lf(x,\frac{S(\pi_{\phi}(f))(x)}
{\lz}\r)\,dx\ls\sum_j
\fai\lf(B_j,\frac{|\lz_j|}{\lz\|\chi_{B_j}\|_{L^\fai(\rn)}}\r),$$
which, together with \eqref{3.2}, implies that $\pi_\phi(f)\in H_{\fai,S}(\rn)$ and
$$\|\pi_{\phi}(f)\|_{H_{\fai,S}(\rn)}\ls
\blz(\{\lz_j\az_j\}_j)\ls\|f\|_{T_{\fai}(\rr^{n+1}_+)}$$
and hence completes the proof of (ii).

Now we prove \eqref{4.7}. For any $x\in\rn$,
by H\"older's inequality, the moment condition of $\az$ and
the Taylor remainder theorem, we see that
\begin{eqnarray}\label{4.8}
\hs\hs\hs S(\az)(x)&&\le\lf\{\int_0^{r_B}\int_{B(x,t)}|\phi_t\ast \az(y)|^2
\,\frac{dy\,dt}{t^{n+1}}\r\}^{1/2}\\ \nonumber
&&\hs+\lf\{\int_{r_B}^{\fz}\int_{B(x,t)}\lf[\int_\rn\frac{1}{t^n}
\lf\{\phi\left(\frac{y-z}{t}\r)-P_{\phi}^s
\left(\frac{y-z}{t}\r)\r\}\az(z)\,dz\r]^2\,\frac{dy\,dt}{t^{n+1}}\r\}^{1/2}\\ \nonumber
&&\le\sum_{j=0}^\fz\lf\{\int_0^{r_B}\int_{B(x,t)}\lf|\phi_t\ast
\lf(\az\chi_{U_j(B)}\r)(y)\r|^2\,\frac{dy\,dt}{t^{n+1}}\r\}^{1/2}\\ \nonumber
&&\hs+\sum_{j=0}^\fz\sum_{\gz\in\zz^n_+,\,|\gz|=s+1}
\lf\{\int_{r_B}^{\fz}\int_{B(x,t)}
\lf[\int_{\rn}\frac{1}{t^n}\r.\r.\\ \nonumber
&&\hs\times\lf|(\partial_x^\gz\phi)\lf(
\frac{\tz(y-z)+(1-\tz)(y-x_B)}{t}\r)\r|\lf|\frac{z-x_B}{t}\r|^{s+1}\\ \nonumber
&&\hs\times\lf|\lf(\az\chi_{U_j(B)}\r)(z)\r|\,dz
\Bigg]^2\,\frac{dy\,dt}{t^{n+1}}\Bigg\}^{1/2}
=:\sum_{j=0}^\fz\lf[\mathrm{E}_j(x)+\mathrm{F}_j(x)\r],
\end{eqnarray}
where $P_{\phi}^s$ denotes the Taylor expansion of $\phi$ about
$(y-x_B)/t$ with degree $s$ and $\tz\in(0,1)$.
For any $j\in\zz_+$, let $B_j:=2^j B$. Then from \eqref{4.8}, the nondecreasing property
of $\fai(x,t)$ in $t$ and Lemma \ref{l2.1}(i), we infer that
\begin{eqnarray}\label{4.9}
&&\int_\rn\fai(x,S(\lz\az)(x))\,dx\\ \nonumber
&&\hs\ls\int_{\rn}\fai\lf(x,|\lz|\sum_{j=0}^\fz
[\mathrm{E}_j(x)+\mathrm{F}_j(x)]\r)\,dx\\ \nonumber
&&\hs\ls\sum_{j=0}^\fz\lf\{
\int_{\rn}\fai\lf(x,|\lz|\mathrm{E}_j(x)\r)\,dx
+\int_{\rn}\fai\lf(x,|\lz|\mathrm{F}_j(x)\r)\,dx\r\}\\ \nonumber
&&\hs\ls\sum_{j=0}^{\fz}\sum_{i=0}^\fz\lf\{\int_{U_i(B_j)}
\fai(x,|\lz|\mathrm{E}_j(x))\,dx+\int_{U_i(B_j)}
\fai(x,|\lz|\mathrm{F}_j(x))\,dx\r\}\\ \nonumber
&&\hs=:\sum_{j=0}^{\fz}\sum_{i=0}^\fz(\mathrm{E}_{i,j}+\mathrm{F}_{i,j}).
\end{eqnarray}

When $i\in\{0,\,\ldots,\,4\}$, by the uniformly upper type 1 and lower type
$p_0$ properties of $\fai$, we see that
\begin{eqnarray}\label{4.10}
\mathrm{E}_{i,j}&&\ls
\|\chi_B\|_{L^\fai(\rn)}\int_{U_i(B_j)}\fai\lf(x,|\lz|
\|\chi_B\|_{L^\fai(\rn)}^{-1}\r)S\lf(\chi_{U_j(B)}\az\r)(x)\,dx\\
\nonumber &&\hs+
\|\chi_B\|_{L^\fai(\rn)}^{p_0}\int_{U_i(B_j)}\fai\lf(x,|\lz|
\|\chi_B\|_{L^\fai(\rn)}^{-1}\r)\lf[S\lf(\chi_{U_j(B)}\az\r)
(x)\r]^{p_0}\,dx\\
\nonumber &&=: \mathrm{G}_{i,j}+\mathrm{H}_{i,j}.
\end{eqnarray}

Now we estimate $\mathrm{G}_{i,j}$. From H\"older's inequality,
the $L^q(\rn)$-boundedness of $S$, $\fai\in\rh_{q'}(\rn)$ and Lemma \ref{l2.4}(vi),
we deduce that
\begin{eqnarray}\label{4.11}
\mathrm{G}_{i,j}&&\ls\|\chi_B\|_{L^\fai(\rn)}
\lf\{\int_{U_i(B_j)}\lf[S\lf(\chi_{U_j(B)}\az\r)(x)\r]^q\,
dx\r\}^{1/q}\\
\nonumber&&\hs\times\lf\{\int_{U_i(B_j)}\lf[\fai\lf(x,|\lz|
\|\chi_B\|_{L^\fai(\rn)}^{-1}\r)\r]^{q'}\,dx\r\}^{1/q'}\\
\nonumber &&\ls\|\chi_B\|_{L^\fai(\rn)}
\|\az\|_{L^q(U_j(B))}|2^{i+j}B|^{-1/q}\fai\lf(2^{i+j}B,|\lz|
\|\chi_B\|_{L^\fai(\rn)}^{-1}\r)\\ \nonumber
&&\ls2^{-j[(n+\epz)-nq_0]}\fai\lf(B,|\lz|
\|\chi_B\|_{L^\fai(\rn)}^{-1}\r).
\end{eqnarray}

For $\mathrm{H}_{i,j}$, by H\"older's inequality, the $L^q(\rn)$-boundedness of $S$
and the fact that $\fai\in\rh_{q'}(\rn)\subset\rh_{(q/p_0)'}(\rn)$, we see that
\begin{eqnarray*}
\mathrm{H}_{i,j}&&\ls\|\chi_B\|^{p_0}_{L^\fai(\rn)}
\lf\{\int_{U_i(B_j)}\lf[S\lf(\chi_{U_j(B)}\az\r)(x)\r]^{q}\,
dx\r\}^{p_0/q}\\
\nonumber&&\hs\times\lf\{\int_{U_i(B_j)}\lf[\fai\lf(x,|\lz|
\|\chi_B\|_{L^\fai(\rn)}^{-1}\r)\r]^{(q/p_0)'}\,dx\r\}^{1/(q/p_0)'}\\
\nonumber &&\ls\|\chi_B\|^{p_0}_{L^\fai(\rn)}
\|\az\|^{p_0}_{L^{q}(U_j(B))}
|2^{i+j}B|^{-p_0/q}\fai\lf(2^{i+j}B,|\lz|
\|\chi_B\|_{L^\fai(\rn)}^{-1}\r)\\ \nonumber
&&\ls2^{-j[(n+\epz)p_0-nq_0]}\fai\lf(B,|\lz|
\|\chi_B\|_{L^\fai(\rn)}^{-1}\r),
\end{eqnarray*}
which, together with \eqref{4.10} and \eqref{4.11}, implies that, for each
$j\in\zz_+$ and $i\in\{0,\,\ldots,\,4\}$,
\begin{eqnarray}\label{4.12}
\mathrm{E}_{i,j}\ls 2^{-j[(n+\epz) p_0-nq_0]}\fai\lf(B,|\lz|
\|\chi_B\|_{L^\fai(\rn)}^{-1}\r).
\end{eqnarray}

When $i\in\nn$ and $i\ge4$, by the uniformly upper type 1 and lower type
$p_0$ properties of $\fai$, we conclude that
\begin{eqnarray}\label{4.13}
\mathrm{E}_{i,j}&&\ls
\|\chi_B\|_{L^\fai(\rn)}\int_{U_i(B_j)}\fai\lf(x,|\lz|
\|\chi_B\|_{L^\fai(\rn)}^{-1}\r)\mathrm{E}_{j}(x)\,dx\\
\nonumber&&\hs+
\|\chi_B\|_{L^\fai(\rn)}^{p_0}\int_{U_i(B_j)}\fai\lf(x,|\lz|
\|\chi_B\|_{L^\fai(\rn)}^{-1}\r)\lf[\mathrm{E}_{j}(x)\r]^{p_0}\,dx
=: \mathrm{K}_{i,j}+\mathrm{J}_{i,j}.
\end{eqnarray}

For any given $x\in U_i(B_j)$ and $y\in B(x,t)$ with $t\in(0,r_B]$,
we see that, for any $z\in U_j(B)$, $|y-z|\gs2^{i+j}r_B$. Then from $\phi\in\cs(\rn)$
and H\"older's inequality, it follows that
\begin{eqnarray*}
\lf|\phi_t\ast\lf(\az\chi_{U_j(B)}\r)(y)\r|&&\ls\int_{U_j(B)}
\frac{t^{\epz}}{(1+|y-z|)^{n+\epz}}|\az(z)|\,dz\\
&&\ls\frac{t^{\epz}}{(2^{i+j}r_B)^{n+\epz}}
\|\az\|_{L^q(U_j(B))}|U_j(B)|^{1/q'},
\end{eqnarray*}
which implies that, for all $x\in U_i(B_j)$,
\begin{eqnarray}\label{4.14}
\mathrm{E}_{j}(x)&&\ls\frac{r_B^{\epz}
\|\az\|_{L^q(U_j(B))}|U_j(B)|^{1/q'}}{(2^{i+j}r_B)^{n+\epz}}
\ls2^{-i(n+\epz)}2^{-j(\epz+\wz\epz)}
\|\chi_B\|^{-1}_{L^\fai(\rn)}.
\end{eqnarray}
By this, H\"older's inequality and Lemma \ref{l2.4}(vi), we see  that
\begin{eqnarray}\label{4.15}
\hs \hs \mathrm{K}_{i,j}&&\ls2^{-i(n+\epz)}
2^{-j(\epz+\wz\epz)}
\fai\lf(2^{i+j}B,|\lz|
\|\chi_B\|_{L^\fai(\rn)}^{-1}\r)\\ \nonumber
&&\ls2^{-i(n+\epz-nq_0)}2^{-j(\epz+\wz\epz-nq_0)}
\fai\lf(B,|\lz|\|\chi_B\|_{L^\fai(\rn)}^{-1}\r).
\end{eqnarray}
Now we estimate $\mathrm{J}_{i,j}$. From \eqref{4.14}
and Lemma \ref{l2.4}(vi), it follows that
\begin{eqnarray}\label{4.16}
\hs\hs \mathrm{J}_{i,j}\ls2^{-ip_0(n+\epz-nq_0/p_0)}
2^{-jp_0(\epz+\wz\epz-nq_0/p_0)}
\fai\lf(B,|\lz|\|\chi_B\|_{L^\fai(\rn)}^{-1}\r).
\end{eqnarray}

By \eqref{4.13}, \eqref{4.15} and \eqref{4.16}, we know that, when $i\in\nn$
with $i\ge4$ and $j\in\zz_+$,
\begin{eqnarray}\label{4.17}
\mathrm{E}_{i,j}\ls2^{-ip_0(n+\epz-nq_0/p_0)}
2^{-jp_0(\epz+\wz\epz-nq_0/p_0)}
\fai\lf(B,|\lz|\|\chi_B\|_{L^\fai(\rn)}^{-1}\r).
\end{eqnarray}

Now we deal with $\mathrm{F}_{i,j}$.
When $i\in\{0,\,\ldots,\,4\}$, similar to the proof of \eqref{4.12}, we see that
\begin{eqnarray}\label{4.18}
\mathrm{F}_{i,j}\ls 2^{-j[(n+\epz)p_0-nq_0]}\fai\lf(B,|\lz|
\|\chi_B\|_{L^\fai(\rn)}^{-1}\r).
\end{eqnarray}
When $i\in\nn$ and $i\ge4$ and $j\in\zz_+$, for any $x\in U_i(B_j)$, $y\in B(x,t)$
with $t\in[r_B, 2^{i+j-2}r_B)$ and $z\in U_j(B)$, we know that $|z-x_B|<2^jr_B$
and  $|y-z|\ge|x-z|-|x-y|\ge2^{i+j-1}r_B-t>2^{i+j-3}r_B$. From these, we deduce that
$$|\tz(y-z)+(1-\tz)(y-x_B)|=|(y-z)-(1-\tz)(z-x_B)|>2^{i+j-4}r_B.
$$
Thus, by this and \eqref{4.1}, together with H\"older's inequality,
we know that, for all $\gz\in\zz_+^n$ with $|\gz|=s+1$,
\begin{eqnarray}\label{4.19}
\qquad\ &&\int_{r_B}^{2^{i+j-2}r_B}\int_{B(x,t)}g(y,t)\frac{dy\,dt}{t^{n+1}}\\ \nonumber
&&\hs\ls\int_{r_B}^{2^{i+j-2}r_B}\int_{B(x,t)}\lf\{
\int_{U_j(B)}\frac{t^{n+s+1+\epz}}{(2^{i+j-4}r_B)
^{n+1+s+\epz}}|z-x_B|^{s+1}|(\az\chi_{U_j(B)})(z)|\,dz\r\}^2\\ \nonumber
&&\hs\hs\times\,\frac{dy\,dt}{t^{2(n+s+1)+n+1}}\\ \nonumber
&&\hs\ls(2^{i+j}r_B)^{-2(n+s+1+\epz)}
(2^jr_B)^{2(s+1)}\|\az\|^2_{L^1(U_j(B))}
\int_{r_B}^{2^{i+j-2}r_B}t^{2\epz-1}\,dt\\ \nonumber
&&\hs\ls(2^{i+j}r_B)^{-2(n+s+1)}
(2^jr_B)^{2(s+1)}\|\az\|^2_{L^q(U_j(B))}|U_j(B)|^{2/q'}\\ \nonumber
&&\hs\ls2^{-2i(n+1+s)}2^{-2j\wz\epz}
\|\chi_B\|^{-2}_{L^\fai(\rn)},
\end{eqnarray}
where
\begin{eqnarray*}
g(y,t):=&&\lf\{\int_\rn\frac{1}{t^n}\lf|(\partial_x^\gz\phi)\lf(
\frac{\tz(y-z)+(1-\tz)(y-x_B)}{t}\r)\r|\r.\\ \nonumber
&&\hs\times\lf.\lf|\frac{z-x_B}{t}\r|^{s+1}
\lf|\lf(\az\chi_{U_j(B)}\r)(z)\r|\,dz\r\}^2.
\end{eqnarray*}
Moreover, when $t\in[2^{i+j-2}r_B,\fz)$, by $\phi\in\cs(\rn)$
and H\"older's inequality, we see that, for all $\gz\in\zz_+^n$
with $|\gz|=s+1$,
\begin{eqnarray*}
&&\int_{2^{i+j-2}r_B}^\fz\int_{B(x,t)}g(y,t)\,\frac{dy\,dt}{t^{n+1}}\\ \nonumber
&&\hs\ls(2^{j}r_B)^{2(s+1)}\|\az\|^2_{L^1(U_j(B))}
\int_{2^{i+j-2}r_B}^\fz t^{-2(n+s+1)-1}\,dt\\ \nonumber
&&\hs\ls(2^{j}r_B)^{2(s+1)}(2^{i+j-2}r_B)^{-2(n+s+1)}
\|\az\|^2_{L^q(U_j(B))}|U_j(B)|^{2/q'}\\ \nonumber
&&\hs\ls2^{-2i(n+s+1)}2^{-2j\wz\epz}
\|\chi_B\|^{-2}_{L^\fai(\rn)},
\end{eqnarray*}
which, together with \eqref{4.19}, implies that, for all $x\in U_i(B_j)$,
\begin{eqnarray}\label{4.20}
\mathrm{F}_j(x)\ls2^{-i(n+s+1)}2^{-j\wz\epz}
\|\chi_B\|^{-1}_{L^\fai(\rn)}.
\end{eqnarray}
Then from \eqref{4.20}, the uniformly lower type $p_0$
property of $\fai$ and Lemma \ref{l2.4}(vi),
it follows that, for each $i\in\nn$ with $i\ge4$ and $j\in\zz_+$,
\begin{eqnarray}\label{4.21}
\mathrm{F}_{i,j}&&\ls\int_{U_i(B_j)}\fai\lf(x,
2^{-i(n+s+1)}2^{-j\wz\epz}|\lz|
\|\chi_B\|^{-1}_{L^\fai(\rn)}\r)\,dx\\ \nonumber
&&\ls2^{-i(n+s+1)p_0}2^{-j\wz\epz p_0}\fai\lf(2^{i+j}B,|\lz|
\|\chi_B\|^{-1}_{L^\fai(\rn)}\r)\\ \nonumber
&&\ls2^{-ip_0(n+s+1-nq_0/p_0)}2^{-jp_0(\wz\epz-nq_0/p_0)}
\fai\lf(B,|\lz|\|\chi_B\|^{-1}_{L^\fai(\rn)}\r).
\end{eqnarray}

Thus, by \eqref{4.9}, \eqref{4.12}, \eqref{4.17},
\eqref{4.18}, \eqref{4.21}, $\epz>n(q_0/p_0-1)$ and $n+1+s>nq_0/p_0$, we conclude that
$$\int_{\rn}\fai\lf(x,|\lz|S(\az)(x)\r)
\,dx\ls\fai\lf(B,|\lz|\|\chi_B\|_{L^\fai(\rn)}^{-1}\r),$$
which implies that \eqref{4.7} holds, and hence
completes the proof of Proposition \ref{p4.1}.
\end{proof}

Recall that one says that $f\in\cs'(\rn)$ \emph{vanishes weakly at infinity}, if for
every $\psi\in\cs(\rn)$, $f\ast\psi_t\to0$ in $\cs'(\rn)$ as $t\to\fz$ (see, for example,
\cite[p.\,50]{fs}). Then we have the following proposition for $H_{\fai,S}(\rn)$.

\begin{proposition}\label{p4.2}
Let $\fai$ be as in Definition \ref{d2.2}, $q\in(1,\fz)$, $s$ be as in Definition
\ref{d4.1} and $\epz\in(nq(\fai)/i(\fai),\fz)$,
where $q(\fai)$ and $i(\fai)$ are respectively as in \eqref{2.3} and \eqref{2.1}.
Assume that $f\in H_{\fai,S}(\rn)$ vanishes weakly at infinity. Then there exist
$\{\lz_j\}_j\subset\cc$ and a sequence $\{\az_j\}_j$ of
$(\fai,\,q,\,s,\,\epz)$-molecules such that
$f=\sum_j\lz_j\az_j$ in both $\cs'(\rn)$ and $H_{\fai,S}(\rn)$.
Moreover, there exists a positive
constant $C$, independent of $f$, such that
$$\blz(\{\lz_j\az_j\}_j):=\inf\lf\{\lz\in(0,\fz):\ \sum_j
\fai\lf(B_j,\frac{|\lz_j|}{\lz\|\chi_{B_j}\|_{L^\fai(\rn)}}\r)\le1\r\}\le
C\|f\|_{H_{\fai,S}(\rn)},$$
where, for each $j$,
$\az_j$ associates with the ball $B_j$.
\end{proposition}

\begin{proof}
By the assumptions of $\phi$ in Definition \ref{4.2}, $f\in\cs'(\rn)$
vanishing weakly at infinity, and \cite[Theorem 1.64]{fs}, we know that
\begin{eqnarray}\label{4.22}
f=\int_0^\fz\phi_t\ast\phi_t\ast f\,\frac{dt}t\ \text{in}\ \cs'(\rn).
\end{eqnarray}
Thus,
$f=\pi_{\phi}(\phi_t\ast f)$ in $\cs'(\rn)$.
Moreover, from $f\in H_{\fai,S}(\rn)$ and Definition \ref{d4.3},
we infer that $\phi_t\ast f\in T_{\fai}
(\rr^{n+1}_+)$, which, together with Theorem \ref{t3.1}, implies that
$\phi_t\ast f=\sum_j\lz_j a_j$ almost everywhere, where $\{\lz_j\}_j$
and $\{a_j\}_j$ are as in \eqref{3.1}. For any $\psi\in\cs(\rn)$,
by using \cite[Theorem 2.3.20]{gra1}, we know that, for any $\epz,\,R
\in(0,\fz)$ with $\epz<R$,
$$\int_\epz^R\int_\rn|\phi_t\ast\phi_t\ast f(x)\psi(x)|\,dx\,\frac{dt}{t}<\fz.$$
From this, \eqref{4.22}, $\phi_t\ast f=\sum_j\lz_j a_j$
and the fact that the collection of sets $\{\supp(a_j)\}_j$ are pairwise disjoint,
up to sets of measure zero (see Remark \ref{r3.x1}), we infer that
\begin{eqnarray}\label{4.22a}
\langle f,\psi\rangle&&=\lim_{R\to\fz,\,\epz\to0}
\lf\langle\int_\epz^R\phi_t\ast\phi_t\ast f\,\frac{dt}t,\psi\r\rangle=
\lim_{R\to\fz,\,\epz\to0}\int_\epz^R
\langle\phi_t\ast\phi_t\ast f,\psi\rangle\,\frac{dt}t\\ \nonumber
&&=\int_0^\fz\langle\phi_t\ast f,\phi_t\ast\psi\rangle\,\frac{dt}t
=\int_0^\fz\int_\rn\lf[\sum_j\lz_ja_j(x,t)\r]\phi_t\ast\psi(x)\,\frac{dx\,dt}{t}\\ \nonumber
&&=\sum_j\lz_j\int_0^\fz\int_\rn a_j(x,t)\phi_t\ast\psi(x)\,\frac{dx\,dt}{t}.
\end{eqnarray}
Moreover, by using H\"older's inequality and Proposition \ref{p4.1}(i),
similar to the proof of \eqref{4.5}, we see that, for any $\psi\in\cs(\rn)$,
$\int_0^\fz\int_\rn |a_j(x,t)\phi_t\ast\psi(x)|\,\frac{dx\,dt}{t}<\fz$,
which, together with \eqref{4.22a}, implies that, for any $\psi\in\cs(\rn)$,
$$\langle f,\psi\rangle=\sum_j\lz_j\int_\rn\pi_\phi(a_j)(x)\psi(x)\,dx.$$
Thus, $f=\sum_j\lz_j\pi_\psi(a_j)$ in $\cs'(\rn)$.
Applying Theorem \ref{t3.1}, Corollary \ref{c3.1} and Proposition
\ref{p4.1}(ii) to $\phi_t\ast f$, we further conclude that
$$f=\pi_{\phi}(\phi_t\ast f)=\sum_j\lz_j
\pi_{\phi}(a_j)=:\sum_j\lz_j\az_j\
\text{in both}\ \cs'(\rn)\ \text{and}\ H_{\fai,S}(\rn)$$
and $\blz(\{\lz_j\az_j\}_j)\ls\|\phi_t\ast f\|_{T_{\fai}(\rr^{n+1}_+)}\sim
\|f\|_{H_{\fai,S}(\rn)}$. Furthermore, by Lemma \ref{l4.1a},
we know that, for each $j$, $\az_j$ is a harmless constant multiple of a $(\fai,q,s,
n+\wz\epz)$-molecule with $\wz\epz>n[q(\fai)/i(\fai)-1]$. Letting $\epz:=n+\wz\epz$,
we then obtain the desired conclusion, which completes the proof of
Proposition \ref{p4.2}.
\end{proof}

To establish the molecular and the Lusin area function
characterizations of $H_\fai(\rn)$, we need the atomic characterization
of $H_\fai(\rn)$ obtained by Ky \cite{k}. We begin with some notions.

\begin{definition}\label{d4.7}
Let $\fai$ be as in Definition \ref{d2.2}.

(I) For each ball $B\subset\rn$, the \emph{space $L^q_\fai(B)$} with
$q\in[1,\fz]$ is defined as the set of all measurable functions
$f$ on $\rn$ supported in $B$ such that
\begin{equation*}
\|f\|_{L^q_{\fai}(B)}:=
\begin{cases}\dsup_{t\in (0,\fz)}
\lf[\frac{1}
{\fai(B,t)}\dint_{\rn}|f(x)|^q\fai(x,t)\,dx\r]^{1/q}<\fz,& q\in [1,\fz),\\
\|f\|_{L^{\fz}(B)}<\fz,&q=\fz.
\end{cases}
\end{equation*}

(II) We say that a triplet $(\fai,\,q,\,s)$ is \emph{admissible},
if $q\in(q(\fai),\fz]$ and $s\in\zz_+$ satisfying that $s\ge\lfz
n[\frac{q(\fai)}{i(\fai)}-1]\rfz$. We say that a measurable function $a$ on
$\rn$ is a \emph{$(\fai,\,q,\,s)$-atom}, if there exists a ball
$B\subset\rn$ such that

$\mathrm{(i)}$ $\supp (a)\subset B$;

$\mathrm{(ii)}$
$\|a\|_{L^q_{\fai}(B)}\le\|\chi_B\|_{L^\fai(\rn)}^{-1}$;

$\mathrm{(iii)}$ $\int_{\rn}a(x)x^{\az}\,dx=0$ for all
$\az\in\zz_+^n$ with $|\az|\le s$.

(III) The  \emph{atomic Musielak-Orlicz Hardy space},
$H^{\fai,\,q,\,s}(\rn)$, is defined as the space of all
$f\in\cs'(\rn)$ satisfying that $f=\sum_jb_j$ in $\cs'(\rn)$,
where $\{b_j\}_j$ is a sequence of multiples of
$(\fai,\,q,\,s)$-atoms with $\supp (b_j)\subset B_j$ and
$\sum_j\fai(B_j,\|b_j\|_{L^q_{\fai}(B_j)})<\fz.$
Moreover, letting
\begin{eqnarray*}
&&\blz_q(\{b_j\}_j):= \inf\lf\{\lz\in(0,\fz):\ \ \sum_j\fai\lf(B_j,
\frac{\|b_j\|_{L^q_{\fai}(B_j)}}{\lz}\r)\le1\r\},
\end{eqnarray*}
the \emph{quasi-norm} of $f\in H^{\fai,\,q,\,s}(\rn)$ is defined
by $\|f\|_{H^{\fai,\,q,\,s}(\rn)}:=\inf\lf\{\blz_q(\{b_j\}_j)\r\}$,
where the infimum is taken over all the
decompositions of $f$ as above.
\end{definition}

\begin{remark}\label{r4.x1}
Let $(\fai,\fz,s)$ be admissible. For any $\epz\in(0,\fz)$
and $q\in(1,\fz)$, by Definitions \ref{d4.7} and \ref{d4.4}, we see that
any $(\fai,\fz,s)$-atom supported in a ball $B$ is a $(\fai,q,s,\epz)$-molecule
associated with the same ball $B$.
\end{remark}

The following lemma is just \cite[Theorem 3.1]{k}.

\begin{lemma}\label{l4.1}
Let $\fai$ be as in Definition \ref{d2.2} and $(\fai,\,q,\,s)$
admissible. Then $H_\fai(\rn)=H^{\fai,\,q,\,s}(\rn)$ with
equivalent norms.
\end{lemma}

Now we state the main theorem of this section as follows.

\begin{theorem}\label{t4.1}
Let $\fai$ be as in Definition \ref{d2.2}.
Assume that $s\in\zz_+$ is as in Definition \ref{d4.2}, $\uc\in(\max
\{n+s,nq(\fai)/i(\fai)\},\fz)$
and $q\in(q(\fai)[r(\fai)]',\fz)$, where $q(\fai)$,
$i(\fai)$ and $r(\fai)$ are,
respectively, as in \eqref{2.3},  \eqref{2.1} and \eqref{2.4}. Then
the following are equivalent:

{\rm(i)} $f\in H_\fai(\rn)$;

{\rm(ii)} $f\in H^{q,s,\uc}_{\fai,\mathrm{mol}}(\rn)$;

{\rm(iii)} $f\in H_{\fai,S}(\rn)$ and $f$ vanishes weakly at infinity.

Moreover, for all $f\in H_\fai(\rn)$, $\|f\|_{H_\fai(\rn)}\sim
\|f\|_{H^{q,s,\uc}_{\fai,\mathrm{mol}}(\rn)}\sim\|f\|_{H_{\fai,S}(\rn)}$, where the
implicit positive constants are independent of $f$.
\end{theorem}

To prove Theorem \ref{t4.1}, we need the following lemma.

\begin{lemma}\label{l4.2}
Let $\fai$ be as in Definition \ref{d2.2}. If $f\in H_\fai(\rn)$, then $f$ vanishes
weakly at infinity.
\end{lemma}

\begin{proof} Observe that, for any $f\in H_\fai(\rn)$, $\phi\in\cs(\rn)$, $x\in\rn$,
$t\in(0,\fz)$ and $y\in B(x,t)$, it holds that
$|f\ast\phi_t(x)|\ls f^\ast(y)$, where $f^\ast$ is as in
Definition \ref{d4.1}.
Hence, since, for any $p\in(0,i(\vz))$,
$\fai$ is of uniformly lower type $p$, then by the
uniformly lower type $p$ and upper type $1$ properties
of $\fai$ and Lemma \ref{l2.1}(iii), we conclude that,
for all $x\in\rn$,
\begin{eqnarray*}
\min\{|f\ast\phi_t(x)|^p,|f\ast\phi_t(x)|\}
&&\ls [\fai(B(x,t),1)]^{-1}\int_{B(x,t)}\fai(y,1)
\min\{[f^\ast(y)]^p,f^\ast(y)\}\,dy\\
&&\ls[\fai(B(x,t),1)]^{-1}\int_{B(x,t)}\vz(y,f^\ast(y))\,dy\\
&&\ls[\fai(B(x,t),1)]^{-1}\max\{\|f\|_{H_\fai(\rn)}^p,\|f\|_{H_\fai(\rn)}\}\to0,
\end{eqnarray*}
as $t\to \fz$. That is, $f$ vanishes weakly at infinity,
which completes the proof of Lemma \ref{l4.2}.
\end{proof}

Now we prove Theorem \ref{t4.1} by using Proposition \ref{p4.2}, and
Lemmas \ref{l4.1} and \ref{l4.2}.

\begin{proof}[Proof of Theorem \ref{t4.1}]
The proof of Theorem \ref{t4.1} is divided into the following three steps.

\textbf{Step I}. $\mathrm{(i)}\Rightarrow\mathrm{(ii)}$.

By Lemma \ref{l4.1}, we see that $H_{\fai}(\rn)=H^{\fai,\,\fz,\,s}(\rn)$. Moreover,
from the definitions of $H^{q,s,\uc}_{ \fai,\mathrm{mol}}(\rn)$
and $H^{\fai,\,\fz,\,s}(\rn)$, together with Remark \ref{r4.x1},
we infer that $H^{\fai,\,\fz,\,s}(\rn)\hookrightarrow
H^{q,s,\uc}_{ \fai,\mathrm{mol}}(\rn)$. Thus,
$H_{\fai}(\rn)\hookrightarrow H^{q,s,\uc}_{ \fai,\mathrm{mol}}(\rn)$, which completes
the proof of Step I.

\textbf{Step II}. $\mathrm{(ii)}\Rightarrow\mathrm{(i)}$.

Let $\az$ be any fixed $(\fai,\,q,\,s,\,\uc)$-molecule associated with a ball
$B:=B(x_B,\,r_B)$. We now prove that $\az$ is an infinite linear combination
of $(\fai,\wz{q},s)$-atoms and $(\fai,\fz,s)$-atoms,
where $\wz{q}$ is determined later such that $(\fai,\wz{q},s)$
is admissible. To this end, for all $k\in\zz_+$, let $\az_{k}:
=\az\chi_{U_k(B)}$ and $\cp_k$ be
the \emph{linear vector space} generated by the  set
$\{x^{\az}\chi_{U_k(B)}\}_{|\az|\le s}$ of polynomials. It is well known (see,
for example, \cite{tw80}) that there exists a unique
\emph{polynomial} $P_k\in \cp_k$ such that, for all
multi-indices $\bz$ with $|\bz|\le s$,
\begin{eqnarray}\label{4.23}
\dint_{\rn}x^\bz\lf[\az_{k}(x)-P_k(x)\r]\,dx=0,
\end{eqnarray}
where $P_k$ is given by the following formula
\begin{eqnarray}\label{4.24}
P_k:=\sum_{\bz\in\zz_+^n,\,|\bz|\le
s}\lf\{\frac{1}{|U_k(B)|}\int_{\rn}
x^{\bz}\,\az_{k}(x)\,dx\r\}\,
Q_{\bz,k}
\end{eqnarray}
and $Q_{\bz,k}$ is the unique polynomial in $\cp_k$
satisfying that, for all multi-indices $\bz$ with $|\bz|\le s$ and the
\emph{dirac function} $\dz_{\gz,\bz}$,
\begin{eqnarray}\label{4.25}
\int_{\rn}x^{\gz}\,Q_{\bz,k}(x)\,dx=|U_k(B)|\,\dz_{\gz,\bz}.
\end{eqnarray}

By the assumption  $q>q(\fai)[r(\fai)]'$, we know that
there exists $\wz q\in(q(\fai),\fz)$ such that $q>\wz q[r(\fai)]'$
and hence $\fai\in\rh_{(\frac{q}{\wz q})'}(\rn)$.
Now we prove that, for each $k\in\zz_+$, $\az_k-P_k$ is a harmless constant multiple
of a $(\fai,\,\wz q,\,s)$-atom and $\sum_{k\in\zz_+}P_k$ can be divided
into an infinite linear combination of $(\fai,\,\fz,\,s)$-atoms.

It was proved in \cite[p.\,83]{tw80} that, for all $k\in\zz_+$,
\begin{eqnarray*}
\sup_{x\in U_k(B)}|P_k(x)|\ls\frac{1}{|U_k(B)|}\|\az_k\|_{L^1(\rn)},
\end{eqnarray*}
which, together with Minkowski's inequality, H\"older's inequality
and Definition \ref{d4.4}(i), implies that
\begin{eqnarray}\label{4.26}
\|\az_k-P_k\|_{L^q(\rn)}&&\ls
\|\az_k\|_{L^q(2^k B)}+\|P_k\|_{L^q(2^kB)}\ls\|\az_k\|_{L^q(U_k(B))}\\ \nonumber
&&\ls2^{-k\uc}|2^kB|^{1/q}\|\chi_B\|^{-1}_{L^\fai(\rn)}.
\end{eqnarray}
From this, H\"older's inequality and $\fai\in\rh_{(\frac{q}{\wz q})'}(\rn)$,
it follows that
\begin{eqnarray*}
&&\lf\{\frac{1}{\fai(2^kB,t)}\int_{2^kB}|\az_k(x)-P_k(x)|^{\wz q}\fai(x,t)
\,dx\r\}^{1/\wz q}\\
&&\hs\ls\frac{1}{[\fai(2^kB,t)]^{1/\wz q}}
\|\az_k-P_k\|_{L^q(2^kB)}
\lf\{\int_{2^kB}[\fai(x,t)]^{(\frac{q}{\wz q})'}\,dx\r\}^{\frac{1}{\wz q
(\frac{q}{\wz q})'}}
\ls2^{-k\uc}\|\chi_B\|^{-1}_{L^\fai(\rn)},
\end{eqnarray*}
which implies that there exists a positive constant $\wz C$ such that, for all $\zz_+$,
\begin{eqnarray}\label{4.27}
\|\az_k-P_k\|_{L^\fai_{\wz q}(2^k B)}\le\wz C
2^{-k\uc}\|\chi_B\|^{-1}_{L^\fai(\rn)}.
\end{eqnarray}
For any $k\in\zz$, let $\mu_k:=\wz C2^{-k\uc}\|\chi_{2^kB}\|_{L^\fai(\rn)}/
\|\chi_{B}\|_{L^\fai(\rn)}$ and
$$a_k:=2^{k\uc}\|\chi_B\|_{L^\fai(\rn)}(\az_k-P_k)/(\wz C\|\chi_{2^kB}\|_{L^\fai(\rn)}).$$
This, combined with \eqref{4.23}, \eqref{4.27} and the fact that $\supp(\az_k-P_k)\subset2^kB$,
implies that, for each $k\in\zz_+$, $a_k$ is a $(\fai,\,\wz q,\,s)$-atom
and $\az_k-P_k=\mu_k a_k$. Moreover, by Minkowski's inequality, \eqref{4.26}
and $\uc>nq(\fai)/i(\fai)\ge n$, we see that
\begin{eqnarray*}
\lf\|\sum_{k\in\zz_+}(\az_k-P_k)\r\|_{L^{q}(\rn)}&&\le\sum_{k\in\zz_+}
\|\az_k-P_k\|_{L^{q}(\rn)}\\
&&\ls\sum_{k\in\zz_+}2^{-k(\uc-n/q)}|B|^{1/q}\|
\chi_B\|^{-1}_{L^\fai(\rn)}\ls|B|^{1/q}\|\chi_B\|^{-1}_{L^\fai(\rn)},
\end{eqnarray*}
which, together with $\az_k-P_k=\mu_k a_k$ for any $k\in\zz_+$,
implies that
\begin{equation}\label{4.27b}
\sum_{k\in\zz_+}(\az_k-P_k)=\sum_{k\in\zz_+}\mu_ka_k \ \text{in}\ L^{q}(\rn).
\end{equation}

Moreover, for any $j\in\zz_+$ and $\ell\in\zz_+^n$, let $$N^j_\ell:=\sum_{k=j}^{\fz}
|U_k(B)|\langle\az_k,x^\ell
\rangle:=\sum_{k=j}^\fz\int_{U_k(B)}\az_k(x)x^\ell\,dx.$$ Then for any
$\ell\in\zz_+^n$ with $|\ell|\le s$, it holds that
\begin{eqnarray}\label{4.27a}
N^0_\ell=\sum_{k=0}^\fz\int_{U_k(B)}\az(x)x^\ell\,dx=0.
\end{eqnarray}
Therefore, by H\"older's inequality and the assumption $\uc\in(n+s,\fz)$, together with
Definition \ref{d4.4}(i), we see
that, for all $j\in\zz_+$ and $\ell\in\zz_+^n$ with $|\ell|\le s$,
\begin{eqnarray}\label{4.28}
\qquad|N^j_\ell|&&\le\sum_{k=j}^\fz\int_{U_k(B)}|\az_j(x)x^\ell|\,dx
\le\sum_{k=j}^\fz(2^kr_B)^{|\ell|}|2^kB|^{1/q'}
\|\az_k\|_{L^q(U_j(B))}\\ \nonumber
&&\le\sum_{k=j}^\fz2^{-k(\uc-n-|\ell|)}|B|^{1+|\ell|/n}
\|\chi_B\|^{-1}_{L^\fai(\rn)}\ls2^{-j(\uc-n-|\ell|)}|B|^{1+|\ell|/n}
\|\chi_B\|^{-1}_{L^\fai(\rn)}.
\end{eqnarray}

Furthermore, from \eqref{4.25} and the homogeneity, we deduce that, for all
$j\in\zz_+$, $\bz\in\zz_+^n$ with $|\bz|\le s$ and $x\in\rn$, $|Q_{\bz,\,j}(x)|
\ls \lf(2^jr_B\r)^{-|\bz|}$,
which, combined with \eqref{4.28}, implies that, for all $j\in\zz_+$, $\ell\in\zz_+^n$
with $|\ell|\le s$ and $x\in\rn$,
\begin{eqnarray}\label{4.29}
|U_j(B)|^{-1}\lf|N^j_\ell Q_{\ell,j}(x)\chi_{U_j(B)}(x)\r|&&\ls2^{-j\uc}
\|\chi_B\|^{-1}_{L^\fai(\rn)}.
\end{eqnarray}
Moreover, by \eqref{4.24} and the definition of $N^j_\ell$, together with
\eqref{4.27a}, we know that
\begin{eqnarray}\label{4.29a}
\qquad\sum_{k=0}^\fz P_k&&=\sum_{\ell\in\zz_+^n,|\ell|\le s}
\sum_{k=0}^\fz\sum_{j=1}^k\langle
\az_j,x^\ell\rangle|U_j(B)|\\ \nonumber
&&=\sum_{\ell\in\zz_+^n,|\ell|\le s}\sum_{k=0}^\fz N^{k+1}_\ell
\lf[|U_k(B)|^{-1}Q_{\ell,k}\chi_{U_k(B)}-
|U_{k+1}(B)|^{-1}Q_{\ell,k+1}\chi_{U_{k+1}(B)}\r]\\ \nonumber
&&=: \sum_{\ell\in\zz_+^n,|\ell|\le s}\sum_{k=0}^\fz b^k_\ell.
\end{eqnarray}
From \eqref{4.29}, it follows that, there exists a positive constant $C_0$ such that,
for all $k\in\zz_+$ and $\ell\in\zz_+^n$ with $|\ell|\le s$,
\begin{eqnarray}\label{4.30}
\|b^k_\ell\|_{L^\fz(\rn)}\le C_02^{-j\uc}
\|\chi_B\|^{-1}_{L^\fai(\rn)}.
\end{eqnarray}
For any $k\in\zz_+$ and $\ell\in\zz_+^n$ with $|\ell|\le s$,
let $\mu^k_\ell:=C_02^{-j\uc}\|\chi_{2^{k+1}B}\|_{L^\fai(\rn)}/\|\chi_{B}\|_{L^\fai
(\rn)}$ and $a^k_\ell:=2^{-j\uc}b^k_\ell\|\chi_{B}\|_{L^\fai(\rn)}/
(C_0\|\chi_{2^{k+1}B}\|_{L^\fai(\rn)})$. Then
$\|a^k_\ell\|_{L^\fz(\rn)}\le\|\chi_{2^{k+1}B}\|^{-1}_{L^\fai(\rn)}$
By \eqref{4.25} and the definitions of $b^k_\ell$ and $a^k_\ell$,
we see that, for all $\gz\in\zz_+^n$ with $|\gz|\le s$,
$\int_{\rn}a^k_\ell(x)x^\gz\,dx=0$. Obviously, $\supp(a^k_\ell)\subset 2^{k+1}B$.
Thus, $a^k_\ell$ is a $(\fai,\fz,s)$-atom and hence
a $(\fai,\wz q,s)$-atom, and $b^k_\ell=\mu^k_\ell a^k_\ell$.
Moreover, similar to \eqref{4.27b}, we see that
$\sum_{k=0}^\fz P_k=
\sum_{\ell\in\zz_+^n,|\ell|\le s}\sum_{k=0}^\fz \mu^k_\ell a^k_\ell$
in $L^q(\rn)$. By this and \eqref{4.27b}, we conclude that
\begin{equation}\label{4.31a}
\az=\sum_{k=0}^\fz(\az_k-P_k)+\sum_{k=0}^\fz P_k
=\sum_{k=0}^\fz\mu_k a_k+\sum_{\ell\in\zz_+^n,|\ell|\le s}
\sum_{k=0}^\fz \mu^k_\ell a^k_\ell
\end{equation}
holds true in $L^q(\rn)$ and hence in $\cs'(\rn)$.

Furthermore, from the
assumption $\uc\in(nq(\fai)/i(\fai),\fz)$, we infer that there exist
$p_0\in(0,i(\fai))$ and $q_0\in(q(\fai),\fz)$ such that $\uc>nq_0/p_0$.
Then $\fai\in\aa_{q_0}(\rn)$ and $\fai$ is of uniformly lower type $p_0$.
By \eqref{4.27}, \eqref{4.30}, the uniformly lower type $p_0$ property of $\fai$,
Lemma \ref{l2.4}(vi) and $\uc>nq_0/p_0$, we conclude that, for all $\lz\in(0,\fz)$,
\begin{eqnarray}\label{4.31}
&&\sum_{k\in\zz_+}\fai\lf(2^kB,\lz\|\mu_k a_k\|_{L^{\wz q}_{\fai}(2^kB)}\r)
+\sum_{|\ell|\le s}\sum_{k\in\zz_+}\fai\lf(2^{k+1}B,\lz
\|\mu^k_\ell a^k_\ell\|_{L^{\wz q}_
{\fai}(2^{k+1}B)}\r)\\ \nonumber
&&\hs\ls\sum_{k\in\zz_+}2^{-p_0k\uc}
\fai\lf(2^{k+1}B,\lz\|\chi_B\|^{-1}_{L^\fai(B)}\r)\\ \nonumber
&&\hs\ls\sum_{k\in\zz_+}2^{-p_0k(\uc-nq_0/p_0)}
\fai\lf(B,\lz\|\chi_B\|^{-1}_{L^\fai(B)}\r)
\ls\fai\lf(B,\lz\|\chi_B\|^{-1}_{L^\fai(B)}\r).
\end{eqnarray}

Let $f\in H^{q,s,\uc}_{\fai,\mathrm{mol}}(\rn)$. Then, by Definition \ref{d4.5}, we know that
there exist $\{\lz_j\}_j\subset\cc$ and a sequence $\{\az_j\}_j$ of
$(\fai,q,s,\uc)$-molecules such that
$f=\sum_{j}\lz_j\az_j$ in $\cs'(\rn)$ and
\begin{eqnarray}\label{4.32}
\|f\|_{H^{q,s,\uc}_{\fai,\mathrm{mol}}(\rn)}\sim
\blz(\{\lz_j\az_j\}_j).
\end{eqnarray}
Then by \eqref{4.31a}, we know that, for
each $j$, there exist $\{\mu_{j,k}\}_k\subset\cc$ and
a sequence $\{a_{j,k}\}_k$ of $(\fai,\wz q,s)$-atoms
such that $\az_j=\sum_{k}\mu_{j,k}a_{j,k}$ in $\cs'(\rn)$.
Thus, $f=\sum_j\sum_k\lz_j\mu_{j,k}a_{j,k}$ in $\cs'(\rn)$, which,
together with Lemma \ref{l4.1}, implies that
$f\in H_{\fai}(\rn)$. Moreover, from \eqref{4.31} and \eqref{4.32}, it follows that
$$\|f\|_{H_{\fai}(\rn)}\ls\blz(\{\lz_j \mu_{j,k}a_{j,k}\}_{j,k})
\ls\blz(\{\lz_j\az_j\}_j)\sim
\|f\|_{H^{q,s,\uc}_{\fai,\mathrm{mol}}(\rn)},$$
which completes the proof of Step II.

\textbf{Step III}. $\mathrm{(ii)}\Leftrightarrow\mathrm{(iii)}$.

Let $f\in H_{\fai,S}(\rn)$ vanishing weakly at infinity. Then from Proposition
\ref{p4.2}, it follows that $f\in H^{q,s,\uc}_{ \fai,\mathrm{mol}}(\rn)$ and
$\|f\|_{H^{q,s,\uc}_{ \fai,\mathrm{mol}}(\rn)}\ls\|f\|_{H_{\fai,S}(\rn)}$.

Conversely, assume that $f\in H^{q,s,\uc}_{ \fai,\mathrm{mol}}(\rn)$. Then by Steps I and
II, we know that $H^{q,s,\uc}_{\fai,\mathrm{mol}}(\rn)=H_\fai(\rn)$ with equivalent
norms, which, together with
Lemma \ref{l4.2}, implies that $f$ vanishes weakly at infinity.  Moreover, from \eqref{4.7},
together with a standard argument, we infer that
$f\in H_{\fai,S}(\rn)$.
This finishes the proof of Step III and hence Theorem \ref{t4.1}.
\end{proof}

\begin{remark}\label{r4.3} By Theorem \ref{t4.1}, we see that
the Musielak-Orlicz Hardy space $H_{\fai,S}(\rn)$ is independent of the choices
of $\phi$ as in Definition \ref{d4.2},
and the Musielak-Orlicz Hardy space $H^{q,s,\uc}_{\fai,\mathrm{mol}}(\rn)$ is independent
of the choices of $q$, $s$ and $\uc$ as in Theorem \ref{t4.1}.
\end{remark}

\section{The Carleson measure characterization of $\mathrm{BMO}_{\fai}(\rn)$\label{s5}}

\hskip\parindent In this section, we first recall the notion of the Musielak-Orlicz
$\bbmo$-type space $\mathrm{BMO}_{\fai}(\rn)$ from \cite{k} and introduce the $\fai$-Carleson
measure. Then we establish the $\fai$-Carleson measure characterization of
$\mathrm{BMO}_{\fai}(\rn)$ by using the Lusin area function characterization
of $H_{\fai}(\rn)$ obtained in Theorem \ref{t4.1}.

The following Musielak-Orlicz $\bbmo$-type space $\bbmo_{\fai}(\rn)$ was
introduced by Ky \cite{k}.

\begin{definition}\label{d5.1}
Let $\fai$ be as in Definition \ref{d2.2}. We say that a locally integrable function
$f$ on $\rn$ is in the  \emph{space $\bbmo_{\fai}(\rn)$}, if
\begin{eqnarray*}
\|f\|_{\bbmo_{\fai}
(\rn)}&:=&\sup_{B\subset\rn}\frac{1}{\|\chi_{B}
\|_{L^{\fai}(\rn)}}
\int_{B}\lf|f(x)-f_B\r|\,dx<\fz,
\end{eqnarray*}
where above and in what follows  the supremum is taken over all the balls
$B\subset\rn$ and
\begin{eqnarray}\label{5.x1}
f_B:=\frac{1}{|B|}\int_B f(y)\,dy.
\end{eqnarray}
\end{definition}

\begin{definition}\label{d5.2}
Let $\fai$ be as in Definition \ref{d2.2}. We say that a measure $d\mu$ on $\rr^{n+1}_+$
is a \emph{$\fai$-Carleson measure}, if
$$\|d\mu\|_{\fai}:=\sup_{B\subset\rn}\frac{|B|^{1/2}}
{\|\chi_B\|_{L^\fai(\rn)}}\lf\{\int_{\widehat{B}}
\lf|d\mu(x,t)\r|\r\}^{1/2}<\fz,
$$
where the supremum is taken over all balls $B\subset\rn$ and $\widehat{B}$
denotes the tent over $B$.
\end{definition}

\begin{theorem}\label{t5.1}
Let $\fai$ be as in Definition \ref{d2.2} and $\phi$ as in Definition \ref{d4.2}.

{\rm (i)} Assume that $b\in\bbmo_\fai(\rn)$ and
$q(\fai)[r(\fai)]'\in(1,2)$. Then
$$d\mu(x,t):=|\phi_t\ast b(x)|^2\frac{dx\,dt}{t}$$
is a $\fai$-Carleson measure on $\rr^{n+1}_+$; moreover,
there exists a positive constant $C$, independent of $b$, such that
$\|d\mu\|_{\fai}\le C\|b\|_{\bbmo_\fai(\rn)}$.

{\rm(ii)} Assume that $nq(\fai)<(n+1)i(\fai)$.
Let $b\in L^2_{\loc}(\rn)$ and
$d\mu(x,t):=|\phi_t\ast b(x)|\frac{dxdt}{t}$
be a $\fai$-Carleson measure
on $\rr^{n+1}_+$. Then $b\in\bbmo_\fai(\rn)$ and, moreover, there exists a
positive constant $C$, independent of $b$, such that
$\|b\|_{\bbmo_\fai(\rn)}\le C\|d\mu\|_{\fai}$.
\end{theorem}

To prove Theorem \ref{t5.1}, we need the following several lemmas.
The following lemma is just \cite[Theorem 2.7]{ly13}.

\begin{lemma}\label{l5.1}
Let $\fai$ be as in Definition \ref{d2.2} and $p\in[1,[q(\fai)]')$, where $q(\fai)$
is as in \eqref{2.3}. Then
$f\in\bbmo_{\fai}(\rn)$ if and only if $f\in \bbmo_\fai^p(\rn)$, where
$$\bbmo_\fai^p(\rn):=\{f\in L^1_{\loc}(\rn):\ \|f\|_{\bbmo_\fai^p(\rn)}<\fz\}$$
and
\begin{eqnarray*}
\|f\|_{\bbmo_\fai^p(\rn)}:=
\sup_{B\subset\rn}\frac{1}{\|\chi_{B}\|_{L^{\varphi}(\rn)}}\left\{
\int_B\left[\frac{|f(x)-f_B|}
{\fai(x,\|\chi_{B}\|_{L^{\fai}(\rn)}^{-1})}\right]^{p}
\fai\left(x,\|\chi_B\|_{L^{\fai}(\rn)}^{-1}\right)\,dx\right\}^{\frac{1}{p}},
\end{eqnarray*}
where the supremum is taken over all balls $B$ in $\rn$ and $f_B$ is as in \eqref{5.x1}.
Moreover, for all $f\in\bbmo_\fai(\rn)$,
$\|f\|_{\bbmo_{\fai}(\rn)}\sim\|f\|_{\bbmo_\fai^p(\rn)}$,
where the implicit constants are independent of $f$.
\end{lemma}

\begin{lemma}\label{l5.2}
Let $\fai$ be as in Theorem \ref{t5.1}, $B_0:=B(x_0,\dz)$ and
$\epz\in(n[\frac{q(\fai)}{i(\fai)}-1],\fz)$, where $q(\fai)$ and $i(\fai)$ are
respectively as in \eqref{2.3} and \eqref{2.1}.
Then there exists a positive constant $C$ such that, for all $f\in\bbmo_{\fai}(\rn)$,
$$\int_\rn\frac{\dz^\epz|f(x)-f_{B_0}|}{\dz^{n+\epz}
+|x-x_0|^{n+\epz}}\,dx\le C\frac{\|\chi_{B_0}\|_{L^\fai(\rn)}}{|B_0|}
\|f\|_{\bbmo_\fai(\rn)},$$
where $f_{B_0}$ is as in \eqref{5.x1} with $B$ replaced by $B_0$.
\end{lemma}

\begin{proof}
For any $k\in\zz_+$, let $B_k:=2^kB_0$. Then for all $k\in\zz_+$, we see that
\begin{eqnarray}\label{5.1}
\qquad|f_{2^{k+1}B}-f_{2^kB}|\ls\frac{1}{|B_{k+1}|}
\int_{B_{k+1}}|f(x)-f_{B_{k+1}}|\,dx\ls
\frac{\|\chi_{B_{k+1}}\|_{L^\fai(\rn)}}{|B_{k+1}|}
\|f\|_{\bbmo_\fai(\rn)}.
\end{eqnarray}
By $\epz\in(n[\frac{q(\fai)}{i(\fai)}-1],\fz)$, we know that
there exist $p_0\in(0,i(\fai))$ and $q_0\in(q(\fai),\fz)$ such that
$\epz>n(\frac{q_0}{p_0}-1)$. Then $\fai\in\aa_{q_0}(\rn)$ and $\fai$ is of uniformly
lower type $p_0$, which, together with Lemmas \ref{l2.4}(vi) and \ref{l2.1}(iii),
implies that, for all $j\in\zz_+$,
$$\fai\lf(B_j, 2^{-jnq_0/p_0}\|\chi_{B_0}\|^{-1}_{L^\fai(\rn)}\r)\ls2^{-jnq_0}
\fai\lf(B_j,\|\chi_{B_0}\|^{-1}_{L^\fai(\rn)}\r)\ls
\fai\lf(B_0,\|\chi_{B_0}\|^{-1}_{L^\fai(\rn)}\r)\sim1.$$
From this, we deduce that, for all $j\in\zz_+$,
\begin{eqnarray}\label{5.x2}
\|\chi_{B_j}\|_{L^\fai(\rn)}
\ls2^{jnq_0/p_0}\|\chi_{B_0}\|_{L^\fai(\rn)},
\end{eqnarray}
which, together with \eqref{5.1}, implies that, for all $k\in\nn$,
\begin{eqnarray*}
|f_{B_k}-f_{B_0}|&&\le\sum_{j=1}^k|f_{B_j}-f_{B_{j-1}}|\ls\|f\|_{\bbmo_\fai(\rn)}
\sum_{j=1}^k\frac{\|\chi_{B_{j}}\|_{L^\fai(\rn)}}{|B_{j}|}\\ \nonumber
&&\ls\lf\{\sum_{j=1}^{k}2^{jn(q_0/p_0-1)}\r\}
\frac{\|\chi_{B_0}\|_{L^\fai(\rn)}}{|B_0|}
\|f\|_{\bbmo_\fai(\rn)}\\ \nonumber
&&\ls2^{kn(q_0/p_0-1)}\frac{\|\chi_{B_0}\|_{L^\fai(\rn)}}{|B_0|}
\|f\|_{\bbmo_\fai(\rn)}.
\end{eqnarray*}
By this and \eqref{5.x2}, together with $\epz>n(q_0/p_0-1)$, we conclude that
\begin{eqnarray*}
&&\int_\rn\frac{\dz^\epz|f(x)-f_{B_0}|}{\dz^{n+\epz}
+|x-x_0|^{n+\epz}}\,dx\\
&&\hs\le \int_{B_0}\frac{\dz^\epz|f(x)-f_{B_0}|}{\dz^{n+\epz}
+|x-x_0|^{n+\epz}}\,dx+\sum_{k=0}^\fz\int_{B_{k+1}\setminus B_{k}}\cdots\\
&&\hs\ls\int_{B_0}\frac{\dz^\epz|f(x)-f_{B_0}|}{\dz^{n+\epz}}\,dx
+\sum_{k=1}^\fz(2^k\dz)^{-(n+\epz)}\dz^\epz\int_{B_k}
|f(x)-f_{B_0}|\,dx\\
&&\hs\ls\frac{\|\chi_{B_0}\|_{L^\fai(\rn)}}{|B_0|}
\|f\|_{\bbmo_\fai(\rn)}+\sum_{k=1}^\fz2^{-k(n+\epz)}\dz^{-n}
\lf[\int_{B_k}|f(x)-f_{B_k}|\,dx+|f_{B_k}-f_{B_0}|\r]\\
&&\hs\ls\lf\{\sum_{k=1}^\fz2^{-k(n+\epz-nq_0/p_0)}\r\}
\frac{\|\chi_{B_0}\|_{L^\fai(\rn)}}{|B_0|}
\|f\|_{\bbmo_\fai(\rn)}\ls\frac{\|\chi_{B_0}\|_{L^\fai(\rn)}}{|B_0|}
\|f\|_{\bbmo_\fai(\rn)},
\end{eqnarray*}
which completes the proof of Lemma \ref{l5.2}.
\end{proof}

Let $H^{\fai,\,\fz,\,s}_{\fin}(\rn)$ denote the sets of
all finite combinations of $(\fai,\,\fz,\,s)$-atoms. It is easy to see,
via the definition of $H^{\fai,\,\fz,\,s}(\rn)$, that
$H^{\fai,\,\fz,\,s}_{\fin}(\rn)$ is dense in $H^{\fai,\,\fz,\,s}(\rn)$
according to the quasi-norm $\|\cdot\|_{H^{\fai,\,\fz,\,s}(\rn)}$.
Recall that, if, for any $f\in H^{\fai,\,\fz,\,s}_{\fin}(\rn)$,
letting
$$\|f\|_{H^{\fai,\,\fz,\,s}_{\fin}(\rn)}:=\inf\lf\{\Lambda_q(\{b_j\}_{j=1}^k):\
f=\sum^k_{j=1}b_j\r\},$$
where the infimum is taken over all finite combinations $\{b_j\}_{j=1}^k$
of $(\fai,\,\fz,\,s)$-atoms of $f$, then Ky proved, in ii) of \cite[Theorem 3.4]{k}, that
$\|\cdot\|_{H^{\fai,\,\fz,\,s}_{\fin}(\rn)}$ and
$\|\cdot\|_{H_\fai(\rn)}$ are equivalent quasi-norms on
$H^{\fai,\,\fz,\,s}_{\fin}(\rn)\cap C(\rn)$ .

The following lemma is just \cite[Theorem 3.2]{k}.

\begin{lemma}\label{l5.3}
Let $\fai$ be as in Definition \ref{d2.2} satisfying $nq(\fai)<(n + 1)i(\fai)$,
where $q(\fai)$ and $i(\fai)$ are respectively as in \eqref{2.3} and \eqref{2.1}.
Then the dual space of $H_\fai(\rn)$, denoted by $(H_\fai(\rn))^\ast$,
is $\mathrm{BMO}_\fai(\rn)$ in the following sense:

{\rm (i)} Suppose that $b\in \bbmo_\fai(\rn)$. Then the linear functional
$L_b:\ f\to L_b(f):=
\int_\rn f(x)b(x)\,dx$, initially defined for $H^{\fai,\,\fz,\,s}_{\fin}(\rn)$,
has a bounded extension to $H_\fai(\rn)$.

{\rm (ii)} Conversely, every continuous linear functional on $H_\fai(\rn)$ arises as the
above with a unique $b\in\bbmo_\fai(\rn)$.

Moreover, $\|b\|_{\bbmo_\fai(\rn)}\sim\|L_b\|_{(H_\fai(\rn))^\ast}$,
where the implicit constants are independent of $b$.
\end{lemma}

\begin{remark}\label{r5.1}
We point out that Ky \cite{k} established Lemma \ref{l5.3} under the additional
assumption that $\fai$ is uniformly locally integrable, namely,
for all compact set $K\subset\rn$, it holds that
$$\int_K\sup_{t>0}\frac{\fai(x,t)}{\int_K\fai(y,t)\,dy}\,dx<\fz.$$
More precisely, Ky \cite{k} needed this additional assumption for $\fai$
in order to guarantee that \cite[Lemma 6.1]{k} holds true. However,
\cite[Lemma 6.1]{k} is an easy consequence of (iv) and (vii) of Lemma \ref{l2.4}.
Thus, this additional assumption is superfluous for Lemma \ref{l5.3}.
\end{remark}

Now we prove Theorem \ref{t5.1} by using Lemmas \ref{l5.2} and \ref{l5.3}.

\begin{proof}[Proof of Theorem \ref{t5.1}]
We first prove (i). Let $b\in\bbmo_\fai(\rn)$.
For any given ball $B_0:=B(x_0,r_0)$, let $\wz B:=2B_0$. Write
\begin{eqnarray}\label{5.2}
b=b_{\wz B}+(b-b_{\wz B})\chi_{\wz B}+
(b-b_{\wz B})\chi_{\rn\setminus\wz B}=:b_1+b_2+b_3.
\end{eqnarray}
For $b_1$, by $\int_\rn\phi(x)\,dx=0$, we see that, for
all $t\in(0,\fz)$, $\phi_t\ast b_1=0$, which implies that
\begin{eqnarray}\label{5.3}
\int_{\widehat{B_0}}|\phi_t\ast b_1(x)|^2\frac{dx\,dt}{t}=0.
\end{eqnarray}
For $b_2$, from Proposition \ref{p4.1}(i), it follows that
\begin{eqnarray}\label{5.1a}
\hs\hs\int_{\widehat{B_0}}|\phi_t\ast b_2(x)|^2\frac{dx\,dt}{t}\le\int_{\rr^{n+1}_+}
|\phi_t\ast b_2(x)|^2\frac{dx\,dt}{t}\ls\|b_2\|^2_{L^2(\rn)}
\sim\int_{\wz B}|b(x)-b_{\wz B}|^2\,dx.
\end{eqnarray}
Moreover, by the assumption $q(\fai)[r(\fai)]'\in(1,2)$, we see that $[q(\fai)]'>2$
and $r(\fai)>\frac{2([q(\fai)]'-1)}{[q(\fai)]'-2}$. From this and the definition of $r(\fai)$,
we infer that there exists $q\in(2,[q(\fai)]')$ such that
$r(\fai)>\frac{2(q-1)}{q-2}$ and hence $\fai\in\rh_{2(q-1)/(q-2)}(\rn)$.
By this, H\"older's inequality and Lemma \ref{l5.1}, we conclude that
\begin{eqnarray*}
\int_{\wz B}|b(x)-b_{\wz B}|^2\,dx&&\le
\lf\{\int_{\wz B}|b(x)-b_{\wz B}|^q\lf[\fai\lf(x,\|\chi_{\wz B}
\|_{L^{\fai}(\rn)}^{-1}\r)\r]^{1-q}\,dx\r\}^{2/q}\\
&&\hs\times\lf\{\int_{\wz B}\lf[\fai\lf(x,
\|\chi_{\wz B}\|_{L^{\fai}(\rn)}^{-1}\r)\r]^{2(q-1)/(q-2)}\,dx\r\}^{(q-2)/q}\\
&&\ls\|b\|^2_{\bbmo_\fai^q(\rn)}\|\chi_{\wz B}\|^2_{L^\fai(\rn)}|B_0|^{-1}
\sim\|b\|^2_{\bbmo_\fai(\rn)}\|\chi_{B_0}\|^2_{L^\fai(\rn)}|B_0|^{-1}
\end{eqnarray*}
which, together with \eqref{5.1a}, implies that
\begin{eqnarray}\label{5.4}
&&\frac{|B_0|}{\|\chi_{B_0}\|_{L^\fai(\rn)}}
\lf\{\frac{1}{|B_0|}\int_{\widehat{B_0}}|\phi_t\ast b_2(x)|^2
\frac{dx\,dt}{t}\r\}^{1/2}\ls\|b\|_{\bbmo_\fai(\rn)}.
\end{eqnarray}

Let $\epz$ be as in Lemma \ref{l5.2}.
For any $(x,t)\in\widehat{B}_0$ and $y\in(\wz{B})^\complement$,
we see that $t\in(0,r_{B_0})$ and $|y-x|\gs|y-x_0|$. By this,
$\phi\in\cs(\rn)$ and Lemma \ref{l5.2}, we see that
\begin{eqnarray*}
|\phi_t\ast b_3(x)|&&\ls
\int_{(\wz{B})^\complement}\frac{t^\epz|b(x)-b_{\wz B}|}{(t+|x-y|)^{n+\epz}}\,dy\\
&&\ls \int_{(\wz{B})^\complement}
\frac{t^\epz|b(x)-b_{\wz B}|}{|y-x_0|^{n+\epz}}\,dy\ls\frac{t^\epz}{r_0^\epz}
\frac{\|\chi_{B_0}\|_{L^\fai(\rn)}}{|B_0|}
\|b\|_{\bbmo_\fai(\rn)},
\end{eqnarray*}
which, together with $\epz>n[q(\fai)/i(\fai)-1]>0$, implies that
\begin{eqnarray*}
&&\frac{|B_0|}{\|\chi_{B_0}\|_{L^\fai(\rn)}}
\lf\{\frac{1}{|B_0|}\int_{\widehat{B_0}}|\phi_t\ast b_3(x)|^2\frac{dx\,dt}{t}
\r\}^{1/2}\\ \nonumber
&&\hs\ls
\lf\{\int_0^{r_0}\frac{t^{2\epz-1}}
{r^{2\epz}_0}\,dt\r\}^{1/2}\|b\|_{\bbmo_\fai(\rn)}
\ls\|b\|_{\bbmo_\fai(\rn)}.
\end{eqnarray*}
From this, \eqref{5.2}, \eqref{5.3} and \eqref{5.4}, we deduce that
$$\frac{|B_0|}{\|\chi_{B_0}\|_{L^\fai(\rn)}}
\lf\{\frac{1}{|B_0|}\int_{\widehat{B_0}}|\phi_t\ast b(x)|^2\frac{dx\,dt}
{t}\r\}^{1/2}\ls\|b\|_{\bbmo_\fai(\rn)},$$
which, together with the arbitrariness of $B_0\subset\rn$, implies
that $d\mu$ is a $\fai$-Carleson measure on $\rr^{n+1}_+$ and
$\|d\mu\|_{\fai}\ls\|b\|_{\bbmo_\fai(\rn)}$. This finishes the proof of (i).

Now we prove (ii). Let $f\in H^{\fai,\,\fz,\,s}_{\fin}(\rn)$.
Then by $f\in L^\fz(\rn)$ with compact support, $b\in L^2_\loc(\rn)$
and the Plancherel formula, together with \eqref{4.x1}, we conclude that
\begin{eqnarray}\label{5.5}
\int_{\rn}f(x)\overline{b(x)}\,dx=\int_{\rr^{n+1}_+}\phi_t\ast
f(x)\overline{\phi_t\ast b(x)}\,\frac{dx\,dt}{t},
\end{eqnarray}
where $\overline{b(x)}$ and $\overline{\phi_t\ast b(x)}$ denote, respectively,
the conjugates of $b(x)$ and $\phi_t\ast b(x)$.
Moreover, from $f\in H^{\fai,\,\fz,\,s}_{\fin}(\rn)$
and Theorem \ref{t4.1}, it follows that $f\in H_{\fai,S}(\rn)$,
which further implies that
$\phi_t\ast f\in T_{\fai}(\rr^{n+1}_+)$.  By this and Theorem
\ref{t3.1}, we find that there exist $\{\lz_j\}_j\subset\cc$ and a sequence
$\{a_j\}_j$ of $(\fai,\fz)$-atoms such that $\phi_t\ast f=\sum_{j}\lz_ja_j$.
From this, \eqref{5.5}, H\"older's inequality, \eqref{3.2}, Theorem \ref{t4.1} and
the uniformly upper type 1 property of $\fai$, we deduce that
\begin{eqnarray*}
\lf|\int_\rn f(x)\overline{b(x)}\,dx\r|
&&\le\sum_j|\lz_j|\int_{\rr^{n+1}_+}|a_j(x,t)|
|\phi_t\ast b(x)|\,\frac{dx\,dt}{t}\\
&&\le\sum_{j}|\lz_j|\lf\{\int_{\widehat{B}_j}|a_j(x,t)|^2
\frac{dx\,dt}{t}\r\}^{1/2}
\lf\{\int_{\widehat{B}_j}|\phi_t\ast b(x)|^2
\frac{dx\,dt}{t}\r\}^{1/2}\\
&&\ls\sum_{j}|\lz_j||B_j|^{1/2}
\|\chi_{B_j}\|^{-1}_{L^\fai(\rn)}
\lf\{\int_{\widehat{B}_j}|\phi_t\ast b(x)|^2
\frac{dx\,dt}{t}\r\}^{1/2}\\
&&\ls\sum_j|\lz_j|\|d\mu\|_{\fai}
\ls\blz(\{\lz_j a_j\}_j)\|d\mu\|_{\fai}\ls\|\phi_t\ast f\|_{T_\fai(\rr^{n+1}_+)}
\|d\mu\|_{\fai}\\
&&\sim\|f\|_{H_{\fai,S}(\rn)}\|d\mu\|_{\fai}\sim\|f\|_{H_\fai(\rn)}\|d\mu\|_{\fai},
\end{eqnarray*}
which implies that $b\in\bbmo_\fai(\rn)$ and $\|b\|_{\bbmo_\fai(\rn)}\ls\|d\mu\|_{\fai}$
and hence completes the proof of Theorem \ref{t5.1}.
\end{proof}

\smallskip

{\bf Acknowledgements.} The authors would like to thank the referee
for her/his many valuable remarks, which essentially improved the presentation
of this article and, particularly, motivated the authors to
try to remove the additional assumption of Theorem \ref{t4.1}
that $\vz\in\rh_2(\rn)$, appeared in the first version of this article.
The authors would also like to thank
Doctor Luong Dang Ky very much for some helpful discussions
on this article, which induce the authors to indeed remove the aforementioned
additional assumption of Theorem \ref{t4.1}.

\bigskip

Shaoxiong Hou, Dachun Yang (Corresponding author) and Sibei Yang

\medskip

School of Mathematical Sciences, Beijing Normal University \&
Laboratory of Mathematics and Complex Systems, Ministry of
Education, Beijing 100875, People's Republic of China

\smallskip

{\it E-mails}: \texttt{houshaoxiong@mail.bnu.edu.cn} (S. Hou)

\hspace{1.55cm}\texttt{dcyang@bnu.edu.cn} (D. Yang)

\hspace{1.55cm}\texttt{yangsibei@mail.bnu.edu.cn} (S. Yang)

\end{document}